\documentclass[11pt]{article}
\usepackage{fontspec}
\usepackage{amsmath,amssymb,amsthm,mathtools,mathrsfs}
\usepackage{microtype,booktabs,longtable,array,enumitem,fvextra,xurl}
\usepackage[margin=1in]{geometry}
\usepackage[dvipsnames]{xcolor}
\usepackage[colorlinks=true,linkcolor=MidnightBlue,citecolor=MidnightBlue,urlcolor=MidnightBlue]{hyperref}
\hypersetup{pdftitle={The irrationality exponent of the lemniscate period},pdfauthor={Gregory V. Chudnovsky},pdfsubject={A determinant proof with a Lean 4 formalization}}
\numberwithin{equation}{section}
\newtheorem{theorem}{Theorem}[section]
\newtheorem{proposition}[theorem]{Proposition}
\newtheorem{lemma}[theorem]{Lemma}
\newtheorem{corollary}[theorem]{Corollary}
\theoremstyle{definition}

\theoremstyle{remark}
\newtheorem{remark}[theorem]{Remark}

\newcommand{\Q}{\mathbb Q}
\newcommand{\Z}{\mathbb Z}
\newcommand{\N}{\mathbb N}

\newcommand{\leanref}[2]{\hyperlink{lean.\csname leanid@\detokenize{#2}\endcsname}{\textsf{[\csname leanid@\detokenize{#2}\endcsname]}}}
\expandafter\def\csname leanid@QuadraticPi.Omega_eventual_rational_lower_bound\endcsname{L001}
\expandafter\def\csname leanid@QuadraticPi.Omega_irrational\endcsname{L002}
\expandafter\def\csname leanid@QuadraticPi.irrationalityExponent_Omega_eq_two\endcsname{L003}
\expandafter\def\csname leanid@QuadraticPi.Omega_approximation_supremum_eq_two\endcsname{L004}
\expandafter\def\csname leanid@QuadraticPi.Lemniscate.Omega\endcsname{L005}
\expandafter\def\csname leanid@QuadraticPi.Lemniscate.integrand_intervalIntegrable\endcsname{L006}
\expandafter\def\csname leanid@QuadraticPi.Lemniscate.Omega_pos\endcsname{L007}
\expandafter\def\csname leanid@QuadraticPi.LemniscatePeriod.squarePair\endcsname{L008}
\expandafter\def\csname leanid@QuadraticPi.LemniscateSquareLattice.square_G_conj\endcsname{L009}
\expandafter\def\csname leanid@QuadraticPi.LemniscateSquareLattice.square_weierstrassP_ofReal_re\endcsname{L010}
\expandafter\def\csname leanid@QuadraticPi.LemniscateSquareLattice.square_derivWeierstrassP_ofReal_re\endcsname{L011}
\expandafter\def\csname leanid@QuadraticPi.LemniscatePeriod.g₃_eq_zero\endcsname{L012}
\expandafter\def\csname leanid@QuadraticPi.LemniscatePeriod.weierstrassP_mul_I\endcsname{L013}
\expandafter\def\csname leanid@QuadraticPi.LemniscateSquareLattice.square_g₂_scale\endcsname{L014}
\expandafter\def\csname leanid@QuadraticPi.LemniscateSquareInvariantPositivity.G_four_re_lower_bound\endcsname{L015}
\expandafter\def\csname leanid@QuadraticPi.LemniscateSquareInvariantPositivity.unit_g₂_re_pos\endcsname{L016}
\expandafter\def\csname leanid@QuadraticPi.LemniscateSquareNormalization.normalizingSide_positive\endcsname{L017}
\expandafter\def\csname leanid@QuadraticPi.LemniscateSquareNormalization.normalizedPair_g₂\endcsname{L018}
\expandafter\def\csname leanid@QuadraticPi.LemniscateSquareNormalization.normalizedPair_g₃\endcsname{L019}
\expandafter\def\csname leanid@QuadraticPi.LemniscateDifferential.order_derivWeierstrassP\endcsname{L020}
\expandafter\def\csname leanid@QuadraticPi.LemniscateDifferential.deriv_derivWeierstrassP\endcsname{L021}
\expandafter\def\csname leanid@QuadraticPi.LemniscateODEUniqueness.second_order_series_unique\endcsname{L022}
\expandafter\def\csname leanid@QuadraticPi.LemniscateODEUniqueness.analytic_second_order_unique\endcsname{L023}
\expandafter\def\csname leanid@QuadraticPi.LemniscateCriticalPoints.two_mul_mem_lattice_of_deriv_eq_zero\endcsname{L024}
\expandafter\def\csname leanid@QuadraticPi.LemniscateRealBranch.real_weierstrassP_tendsto_atTop\endcsname{L025}
\expandafter\def\csname leanid@QuadraticPi.LemniscateRealBranch.square_real_deriv_ne_zero\endcsname{L026}
\expandafter\def\csname leanid@QuadraticPi.LemniscateRealBranch.one_lt_realP\endcsname{L027}
\expandafter\def\csname leanid@QuadraticPi.LemniscateRealBranch.realP_halfPeriod_eq_one\endcsname{L028}
\expandafter\def\csname leanid@QuadraticPi.LemniscateRealBranch.realDerivP_neg\endcsname{L029}
\expandafter\def\csname leanid@QuadraticPi.LemniscateRealCalculus.hasDerivAt_realP\endcsname{L030}
\expandafter\def\csname leanid@QuadraticPi.LemniscateRealCalculus.hasDerivAt_realDerivP_of_normalization\endcsname{L031}
\expandafter\def\csname leanid@QuadraticPi.LemniscateRealCalculus.real_cubic_of_normalization\endcsname{L032}
\expandafter\def\csname leanid@QuadraticPi.Lemniscate.reciprocal_sqrt_hasDerivAt\endcsname{L033}
\expandafter\def\csname leanid@QuadraticPi.Lemniscate.halfPeriod_eq_integral_of_branch\endcsname{L034}
\expandafter\def\csname leanid@QuadraticPi.Lemniscate.period_eq_Omega_of_real_branch\endcsname{L035}
\expandafter\def\csname leanid@QuadraticPi.LemniscateOmegaNormalization.omegaPair_g₂_eq_four\endcsname{L036}
\expandafter\def\csname leanid@QuadraticPi.LemniscateOmegaNormalization.period_eq_Omega_of_normalization\endcsname{L037}
\expandafter\def\csname leanid@QuadraticPi.LemniscateOmegaNormalization.normalizingSide_eq_Omega\endcsname{L038}
\expandafter\def\csname leanid@QuadraticPi.LemniscatePeriod.omega_diagonalHalfPeriod_not_mem\endcsname{L039}
\expandafter\def\csname leanid@QuadraticPi.LemniscatePeriod.omega_weierstrassP_diagonalHalfPeriod\endcsname{L040}
\expandafter\def\csname leanid@QuadraticPi.LemniscatePeriod.omega_derivWeierstrassP_diagonalHalfPeriod\endcsname{L041}
\expandafter\def\csname leanid@QuadraticPi.LemniscateLocal.localX\endcsname{L042}
\expandafter\def\csname leanid@QuadraticPi.LemniscateLocal.ellipticLog\endcsname{L043}
\expandafter\def\csname leanid@QuadraticPi.LemniscateLocal.localX_equation\endcsname{L044}
\expandafter\def\csname leanid@QuadraticPi.LemniscateLocal.ellipticLog_derivative_equation\endcsname{L045}
\expandafter\def\csname leanid@QuadraticPi.LemniscateLocal.ellipticLog_coeff_degree_cleared\endcsname{L046}
\expandafter\def\csname leanid@QuadraticPi.LemniscateTaylorBridge.omegaXSeries\endcsname{L047}
\expandafter\def\csname leanid@QuadraticPi.LemniscateTaylorBridge.omegaYSeries\endcsname{L048}
\expandafter\def\csname leanid@QuadraticPi.LemniscateOmegaGerm.omega_series_cubic\endcsname{L049}
\expandafter\def\csname leanid@QuadraticPi.LemniscateOmegaGerm.omegaYSeries_derivative\endcsname{L050}
\expandafter\def\csname leanid@QuadraticPi.LemniscateOmegaGerm.omega_localX_subst_eq\endcsname{L051}
\expandafter\def\csname leanid@QuadraticPi.LemniscateOmegaGerm.omega_ellipticLog_subst_eq\endcsname{L052}
\expandafter\def\csname leanid@QuadraticPi.LemniscateOmegaGerm.omegaYSeries_linear\endcsname{L053}
\expandafter\def\csname leanid@QuadraticPi.LemniscateTaylorBridge.omegaXSeries_derivative\endcsname{L054}
\expandafter\def\csname leanid@QuadraticPi.LemniscateLogParameterChart.logParameter\endcsname{L055}
\expandafter\def\csname leanid@QuadraticPi.LemniscateLogParameterChart.xFunction\endcsname{L056}
\expandafter\def\csname leanid@QuadraticPi.LemniscateLogParameterChart.yFunction\endcsname{L057}
\expandafter\def\csname leanid@QuadraticPi.LemniscateLogParameterChart.xSeries\endcsname{L058}
\expandafter\def\csname leanid@QuadraticPi.LemniscateLogParameterChart.ySeries\endcsname{L059}
\expandafter\def\csname leanid@QuadraticPi.LemniscateLogParameterChart.ellSeries\endcsname{L060}
\expandafter\def\csname leanid@QuadraticPi.LemniscateOmegaGerm.log_series_cubic\endcsname{L061}
\expandafter\def\csname leanid@QuadraticPi.LemniscateOmegaGerm.log_ySeries_derivative\endcsname{L062}
\expandafter\def\csname leanid@QuadraticPi.LemniscateOmegaGerm.log_localX_subst\endcsname{L063}
\expandafter\def\csname leanid@QuadraticPi.LemniscateOmegaGerm.log_ellipticLog_subst\endcsname{L064}
\expandafter\def\csname leanid@QuadraticPi.LemniscateOmegaGerm.log_ySeries_linear\endcsname{L065}
\expandafter\def\csname leanid@QuadraticPi.LemniscateLogParameterChart.ellSeries_linear\endcsname{L066}
\expandafter\def\csname leanid@QuadraticPi.SeparatedWeightPA.BaseData\endcsname{L067}
\expandafter\def\csname leanid@QuadraticPi.LemniscateInterpolation.formalJet_cubic_relation\endcsname{L068}
\expandafter\def\csname leanid@QuadraticPi.LemniscateInterpolation.Statement\endcsname{L069}
\expandafter\def\csname leanid@QuadraticPi.LemniscateGlobalInterpolation.ellipticSeparatedWeightPA\endcsname{L070}
\expandafter\def\csname leanid@QuadraticPi.LemniscatePlaceContact.contact\endcsname{L071}
\expandafter\def\csname leanid@QuadraticPi.LemniscateGrowthBudget.strict_growth_contact_margin\endcsname{L072}
\expandafter\def\csname leanid@QuadraticPi.LemniscateGrowthBudget.growth_factor_lt_three_halves\endcsname{L073}
\expandafter\def\csname leanid@QuadraticPi.LemniscateComparisonConstant.exists_thresholds\endcsname{L074}
\expandafter\def\csname leanid@QuadraticPi.LemniscateWeightGeometry.NumericData\endcsname{L075}
\expandafter\def\csname leanid@QuadraticPi.LemniscateGeometryFrame.formalJet_polynomialFrame\endcsname{L076}
\expandafter\def\csname leanid@QuadraticPi.LemniscateWeightedDerivation.polynomialFrame_supportBound\endcsname{L077}
\expandafter\def\csname leanid@QuadraticPi.LemniscateAuxiliaryPolynomial.eventually_exists_nonzero_elliptic_auxiliary\endcsname{L078}
\expandafter\def\csname leanid@QuadraticPi.LemniscateCurveAuxiliary.exists_auxiliary_with_persistence_inputs\endcsname{L079}
\expandafter\def\csname leanid@QuadraticPi.LemniscateNormalComparison.normal_product_comparison\endcsname{L080}
\expandafter\def\csname leanid@QuadraticPi.LemniscateDifferentialRigidity.finite_extension_no_exact_cubicDifferential\endcsname{L081}
\expandafter\def\csname leanid@QuadraticPi.LemniscateCurveRigidity.constant_projection\endcsname{L082}
\expandafter\def\csname leanid@QuadraticPi.LemniscateCurveInequality.weighted_curve_inequality\endcsname{L083}
\expandafter\def\csname leanid@QuadraticPi.LemniscateFixedJetScale.exists_scale\endcsname{L084}
\expandafter\def\csname leanid@QuadraticPi.LemniscateJetSelector.localized_selector_eq\endcsname{L085}
\expandafter\def\csname leanid@QuadraticPi.LemniscateJetSelector.selectorIdeal_le_prime_iff\endcsname{L086}
\expandafter\def\csname leanid@QuadraticPi.LemniscateCurveModel.hyperplane_degree\endcsname{L087}
\expandafter\def\csname leanid@QuadraticPi.LemniscateCurveModel.exceptional_degree\endcsname{L088}
\expandafter\def\csname leanid@QuadraticPi.LemniscateBlowupGeometry.interpolationBundle_ample\endcsname{L089}
\expandafter\def\csname leanid@QuadraticPi.LemniscateGlobalInterpolation.actual_blowupBundle_ample\endcsname{L090}
\expandafter\def\csname leanid@QuadraticPi.LemniscateGlobalInterpolation.eventual_packetMap_of_ample\endcsname{L091}
\expandafter\def\csname leanid@QuadraticPi.LemniscateGlobalInterpolation.eventual_divisible_packetMap_of_ample\endcsname{L092}
\expandafter\def\csname leanid@QuadraticPi.LemniscateArithmetic.coefficient_monomialImage\endcsname{L093}
\expandafter\def\csname leanid@QuadraticPi.LemniscateInterpolationMatrix.formal_packets_surjective_iff_truncated\endcsname{L094}
\expandafter\def\csname leanid@QuadraticPi.LemniscateInterpolationMatrix.matrix_surjective_of_epa_packet\endcsname{L095}
\expandafter\def\csname leanid@QuadraticPi.LemniscateUnconditionalMinors.cofinal_minors\endcsname{L096}
\expandafter\def\csname leanid@QuadraticPi.LemniscateLocal.ellipticLog_trunc_lcm_cleared\endcsname{L097}
\expandafter\def\csname leanid@QuadraticPi.LemniscateDeterminant.selectedMinor_entries_gaussian\endcsname{L098}
\expandafter\def\csname leanid@QuadraticPi.LemniscateDeterminant.selectedMinor_clearing_bound\endcsname{L099}
\expandafter\def\csname leanid@QuadraticPi.LemniscateDeterminant.selectedMinor_arithmetic_lower_bound\endcsname{L100}
\expandafter\def\csname leanid@QuadraticPi.LemniscateUnconditionalMinors.cofinal_arithmetic_lower_bounds\endcsname{L101}
\expandafter\def\csname leanid@QuadraticPi.LemniscateAnalyticRowData.row_order_sum_le\endcsname{L102}
\expandafter\def\csname leanid@QuadraticPi.LemniscateExponentConditional.mean_bounds\endcsname{L103}
\expandafter\def\csname leanid@QuadraticPi.LemniscateExponentConditional.minor\endcsname{L104}
\expandafter\def\csname leanid@QuadraticPi.LemniscateLogJetDeterminant.det_fullLog_selected_eq_selectedMinor\endcsname{L105}
\expandafter\def\csname leanid@QuadraticPi.LemniscateAnalytic.squareTheta\endcsname{L106}
\expandafter\def\csname leanid@QuadraticPi.LemniscateAnalytic.squareTheta_differentiable\endcsname{L107}
\expandafter\def\csname leanid@QuadraticPi.LemniscateAnalytic.omegaMultiplier\endcsname{L108}
\expandafter\def\csname leanid@QuadraticPi.LemniscateAnalytic.squareTheta_add_one\endcsname{L109}
\expandafter\def\csname leanid@QuadraticPi.LemniscateAnalytic.squareTheta_add_I\endcsname{L110}
\expandafter\def\csname leanid@QuadraticPi.LemniscateAnalytic.squareTheta_half_zero\endcsname{L111}
\expandafter\def\csname leanid@QuadraticPi.LemniscateAnalytic.omegaMultiplier_lattice_zero\endcsname{L112}
\expandafter\def\csname leanid@QuadraticPi.LemniscateAnalytic.gaussianThetaNorm_add_gaussian\endcsname{L113}
\expandafter\def\csname leanid@QuadraticPi.LemniscateAnalytic.omegaMultiplier_uniform_nonzero_inverse_pow\endcsname{L114}
\expandafter\def\csname leanid@QuadraticPi.LemniscateAnalytic.omegaClearedX_entire\endcsname{L115}
\expandafter\def\csname leanid@QuadraticPi.LemniscateAnalytic.omegaClearedY_entire\endcsname{L116}
\expandafter\def\csname leanid@QuadraticPi.LemniscateAnalytic.omega_cleared_quadratic_growth\endcsname{L117}
\expandafter\def\csname leanid@QuadraticPi.LemniscateAnalytic.omegaMonomialColumn_uniform_growth\endcsname{L118}
\expandafter\def\csname leanid@QuadraticPi.LemniscateAnalytic.omegaThetaArgument_im_sq_le\endcsname{L119}
\expandafter\def\csname leanid@QuadraticPi.LemniscateAnalytic.omegaMonomialColumn_entire\endcsname{L120}
\expandafter\def\csname leanid@QuadraticPi.LemniscateAnalytic.omegaMonomialColumn_eq\endcsname{L121}
\expandafter\def\csname leanid@QuadraticPi.LemniscateActualMinorTransport.clearedMinor_eq_transformed\endcsname{L122}
\expandafter\def\csname leanid@QuadraticPi.LemniscateActualMinorTransport.norm_clearedMinor_eq\endcsname{L123}
\expandafter\def\csname leanid@QuadraticPi.LemniscateAnalytic.omegaInverseConstant_bound\endcsname{L124}
\expandafter\def\csname leanid@QuadraticPi.LemniscateIrrationalityExponent.norm_actual_minor_le\endcsname{L125}
\expandafter\def\csname leanid@QuadraticPi.LemniscateDisplacementExpansion.displacedColumn_expansion\endcsname{L126}
\expandafter\def\csname leanid@QuadraticPi.LemniscateTranslatedDeterminant.det_displacedMatrix_expansion\endcsname{L127}
\expandafter\def\csname leanid@QuadraticPi.LemniscateDisplacement.gridDelta_exp_bound\endcsname{L128}
\expandafter\def\csname leanid@QuadraticPi.LemniscateDisplacement.norm_rowScalar_exp_weighted\endcsname{L129}
\expandafter\def\csname leanid@QuadraticPi.LemniscateTranslatedDeterminant.scalar_product_bound\endcsname{L130}
\expandafter\def\csname leanid@QuadraticPi.LemniscateAnalytic.omegaTagColumn_uniform_growth\endcsname{L131}
\expandafter\def\csname leanid@QuadraticPi.LemniscateAnalytic.omegaTaggedDeterminant_exp_bound\endcsname{L132}
\expandafter\def\csname leanid@OAI.PiExponent.AnalyticCollision.translated_row_determinant_exp_bound\endcsname{L133}
\expandafter\def\csname leanid@QuadraticPi.LemniscateTranslatedDeterminant.displaced_determinant_log_bound\endcsname{L134}
\expandafter\def\csname leanid@QuadraticPi.LemniscateSameMinorBound.fixedCost\endcsname{L135}
\expandafter\def\csname leanid@QuadraticPi.LemniscateAnalytic.final_holomorphic_budget\endcsname{L136}
\expandafter\def\csname leanid@QuadraticPi.LemniscateSameMinorBound.explicit_bound_of_norm_transfer\endcsname{L137}
\expandafter\def\csname leanid@OAI.PiExponent.MatrixCounting.tendsto_rowCount_normalized\endcsname{L138}
\expandafter\def\csname leanid@OAI.PiExponent.MatrixCounting.tendsto_lowIndexCount_normalized\endcsname{L139}
\expandafter\def\csname leanid@QuadraticPi.LemniscateAnalyticRemainder.tendsto_remainder\endcsname{L140}
\expandafter\def\csname leanid@QuadraticPi.LemniscateExponentConditional.tendsto_collisionRate\endcsname{L141}
\expandafter\def\csname leanid@QuadraticPi.LemniscateIrrationalityExponent.sameMinorAnalyticStatement\endcsname{L142}
\expandafter\def\csname leanid@OAI.PiExponent.exists_parameters\endcsname{L143}
\expandafter\def\csname leanid@QuadraticPi.LemniscateSquareParameters.exists_squareDimension_margin\endcsname{L144}
\expandafter\def\csname leanid@QuadraticPi.LemniscateParameterMargin.exists_margin\endcsname{L145}
\expandafter\def\csname leanid@QuadraticPi.LemniscateDataConstruction.exists_scalarData\endcsname{L146}
\expandafter\def\csname leanid@QuadraticPi.LemniscateDataConstruction.exists_approximations_above_thresholds\endcsname{L147}
\expandafter\def\csname leanid@QuadraticPi.LemniscateDataConstruction.exists_small_error_approximations\endcsname{L148}
\expandafter\def\csname leanid@QuadraticPi.cofinal_determinant_bounds_inconsistent\endcsname{L149}
\expandafter\def\csname leanid@QuadraticPi.LemniscateExponentConditional.no_unbounded_approximations_of_analyticStatement\endcsname{L150}
\expandafter\def\csname leanid@QuadraticPi.irrationalityExponent\endcsname{L151}
\expandafter\def\csname leanid@QuadraticPi.irrational_of_eventualLowerBound\endcsname{L152}
\expandafter\def\csname leanid@QuadraticPi.two_mem_approximationExponents\endcsname{L153}
\expandafter\def\csname leanid@QuadraticPi.irrationalityExponent_eq_two_of_eventualLowerBound\endcsname{L154}
\title{The irrationality exponent of the lemniscate period\\[3pt]
\large A determinant proof with a Lean 4 formalization}
\author{Gregory V. Chudnovsky\thanks{Developed in collaboration with OpenAI ChatGPT, including mathematical exposition and formalization assistance. The accompanying Lean source files contain the formal proof.}}
\date{October 10, 2026}
\begin{document}
\maketitle
\begin{abstract}
Let $\Omega=2\int_0^1(1-t^4)^{-1/2}\,dt$. We prove that, for every
real $\nu>2$, there is $Q\geq2$ such that
$q^{-\nu}<|\Omega-p/q|$ for every integer $p$ and every integer $q\geq Q$.
Consequently $\Omega$ is irrational and its irrationality exponent is two.
The proof combines elliptic separated-weight interpolation, arithmetic
denominator estimates, and an analytic determinant bound on a square
complex-multiplication grid. The determinant comparison uses the same
selected source columns on both sides. We give the integral-to-lattice
normalization, the exact local-coordinate and common-multiplier transport,
and the parameter order that makes the two bounds incompatible.
The accompanying Lean 4 development proves the endpoint without additional
hypotheses. Numbered formal references link the mathematical argument to
exact declarations in the complete source distribution; an appendix gives
the pinned build environment and the dependency audit.
\end{abstract}
\tableofcontents
\clearpage

\section{The theorem and the structure of the proof}
\label{sec:main}

The period considered in this paper is fixed by the integral
\begin{equation}\label{eq:Omega-main}
  \Omega=2\int_0^1\frac{dt}{\sqrt{1-t^4}}.
\end{equation}
It is positive and finite. The factor two is part of the definition.
The identification with the square-lattice Weierstrass model is proved
in Section~\ref{norm:section}; it is not an independent normalization
assumption. Classical lemniscatic functions and their complex multiplication
provide the background for this choice of period \cite{coxhyde}.

For an irrational real number $\xi$, write
\begin{equation}\label{eq:mu-main}
 \mu(\xi)=\sup\left\{\nu>0:
 \begin{array}{l}
 \text{infinitely many distinct reduced }p/q\in\Q,\ q\geq2,\\[-1pt]
 0<|\xi-p/q|<q^{-\nu}
 \end{array}\right\}.
\end{equation}
We prove irrationality as well as the upper approximation bound, so the
positive-error convention in \eqref{eq:mu-main} introduces no exceptional
rational case. The Lean definition is a real supremum; the proof establishes
both boundedness above and the Dirichlet lower bound.

\begin{theorem}\label{thm:main}
For every real $\nu>2$ there is an integer $Q\geq2$ such that
\begin{equation}\label{eq:strict-main}
 q^{-\nu}<\left|\Omega-\frac pq\right|
 \qquad(p\in\Z,\ q\in\N,\ q\geq Q).
\end{equation}
In particular $\Omega$ is irrational and $\mu(\Omega)=2$.
\end{theorem}

The formal endpoints are
\leanref{LemniscateIrrationalityExponent}{QuadraticPi.Omega_eventual_rational_lower_bound},
\leanref{LemniscateIrrationalityExponent}{QuadraticPi.Omega_irrational},
and
\leanref{LemniscateIrrationalityExponent}{QuadraticPi.irrationalityExponent_Omega_eq_two}.
The ordinary reduced-rational supremum is exposed, with its defining set
written out, in
\leanref{LemniscateIrrationalityExponent}{QuadraticPi.Omega_approximation_supremum_eq_two}.

\subsection{Why the same determinant matters}

Assume, to obtain a contradiction, that approximations with error at most
$q^{-\nu}$ occur at arbitrarily large denominators for some $\nu>2$.
One first fixes the interpolation dimension and a square grid, and then
chooses finitely many of those approximations with successively separated
weights $w_i=\lceil\log q_i\rceil$. Elliptic interpolation supplies a nonzero
arithmetic minor at arbitrarily large interpolation levels $H$.
Denominator clearing bounds this minor from below. Complex analysis bounds
the very same minor from above. The strict parameter margins then give a
contradiction.

There are three exact identifications behind this comparison. The integral
\eqref{eq:Omega-main} determines the actual lattice invariant $g_2=4$.
The local elliptic parameter at the diagonal half-period has linear
coefficient $-1$ after the substitution $s=\log(1+u)$. Finally, multiplying
all columns by a common entire factor changes the finite-jet determinant by
an explicit product of its nonzero constant coefficients. Thus neither a
change of period nor a change of selected columns is hidden in the analytic
estimate. Constant transverse displacements are expanded by a finite
binomial identity, before the collision bound is applied.

\subsection{Organization and formal references}

Section~\ref{norm:section} proves the period normalization and the local
analytic/formal comparison. The subsequent sections establish elliptic
separated-weight interpolation, the arithmetic lower bound, the analytic
upper bound, and the final strict comparison. The generic graded-jet,
cohomological, and collision-counting infrastructure is supplied by the
vendored OAI library \cite{oaiformalsnapshot}; the elliptic theorem and
the period-specific bridges are proved in the accompanying project.

References of the form \textsf{[L001]} are clickable links to the exact
Lean declaration index in Appendix~\ref{app:lean-index}. That index gives
each fully qualified name, its source module, and its declaration line in
the frozen source release. These are formal source references, distinct
from the bibliography. The formalization uses Lean 4 \cite{lean4} and
Mathlib \cite{mathlib}. Appendix~\ref{app:formal} records what was checked
and gives commands to reproduce the result.

The result of this paper concerns precisely the period
\eqref{eq:Omega-main}. All approximation constants and denominator
thresholds may depend on the already fixed exponent $\nu$.

\section{The integral-defined period and its actual elliptic germ}
\label{norm:section}

The arithmetic construction uses the curve
\begin{equation}\label{norm:curve}
  E:\quad y^2=x^3-x
\end{equation}
with local parameter $t=y$ at $(0,0)$.  Its identification with the
analytic functions used in the determinant estimate requires the precise
normalization of the period.  We establish this identification from the
defining integral, including the factor of two and the sign of the local
parameter.

\subsection{A positive square-lattice normalization}

Throughout this section,
\begin{equation}\label{norm:Omega}
  \Omega=2\int_0^1\frac{dt}{\sqrt{1-t^4}}.
\end{equation}
The integral converges: on $0\leq t<1$ one has $t^4\leq t$, hence
\[
  0<\frac{1}{\sqrt{1-t^4}}
    \leq\frac{1}{\sqrt{1-t}}.
\]
The right-hand side is integrable, and positivity on $(0,1)$ gives
$\Omega>0$.  Lean treats this as an interval integral; the values assigned
at its endpoints do not affect the integral.
The precise definitions and these results are
\leanref{LemniscateTarget}{QuadraticPi.Lemniscate.Omega},
\leanref{LemniscateTarget}{QuadraticPi.Lemniscate.integrand_intervalIntegrable},
and \leanref{LemniscateTarget}{QuadraticPi.Lemniscate.Omega_pos}.

For $r>0$, put
\[
  \Lambda_r=r(\mathbb Z+i\mathbb Z),\qquad
  \wp_r=\wp_{\Lambda_r},\qquad
  G_k(\Lambda_r)=\sum_{\lambda\in\Lambda_r\setminus\{0\}}\lambda^{-k},
\]
and use the Weierstrass convention
$g_2=60G_4$, $g_3=140G_6$ and
$(\wp')^2=4\wp^3-g_2\wp-g_3$.
The period pair has generators $r$ and $ir$, rather than $2r$ and $2ir$.
It is defined by
\leanref{LemniscatePeriodGeometry}{QuadraticPi.LemniscatePeriod.squarePair}.
Conjugation preserves $\Lambda_r$, so $g_2(\Lambda_r)$, $\wp_r(t)$,
and $\wp'_r(t)$ are real for real $t$ off the lattice.  Rotation by $i$
gives $G_6=-G_6$, hence $g_3(\Lambda_r)=0$, and gives
$\wp_r(iz)=-\wp_r(z)$.  Direct reindexing of the convergent lattice sums
also yields
\begin{equation}\label{norm:scaling}
  g_2(\Lambda_r)=r^{-4}g_2(\Lambda_1).
\end{equation}
These facts are recorded in
\leanref{LemniscateSquareLattice}{QuadraticPi.LemniscateSquareLattice.square_G_conj},
\leanref{LemniscateSquareLattice}{QuadraticPi.LemniscateSquareLattice.square_weierstrassP_ofReal_re},
\leanref{LemniscateSquareLattice}{QuadraticPi.LemniscateSquareLattice.square_derivWeierstrassP_ofReal_re},
\leanref{LemniscatePeriodGeometry}{QuadraticPi.LemniscatePeriod.g₃_eq_zero},
\leanref{LemniscatePeriodGeometry}{QuadraticPi.LemniscatePeriod.weierstrassP_mul_I},
and \leanref{LemniscateSquareLattice}{QuadraticPi.LemniscateSquareLattice.square_g₂_scale}.

\begin{lemma}[Positivity without a period evaluation]\label{norm:positive}
The unit square lattice satisfies
\[
  \operatorname{Re}G_4(\Lambda_1)
    \geq 4-\frac{\pi^4}{36}>0.
\]
In particular $g_2(\Lambda_1)>0$.
\end{lemma}
\begin{proof}
For $a,b\in\mathbb Z$ with $ab\ne0$,
\[
  \bigl|(a+ib)^{-4}\bigr|
   =\frac{1}{(a^2+b^2)^2}\leq\frac{1}{4a^2b^2}.
\]
Define $e(a,b)=1/(4a^2b^2)$ when $ab\ne0$ and $e(a,b)=0$
otherwise.  Then
$\operatorname{Re}(a+ib)^{-4}+e(a,b)\geq0$; the origin is omitted,
and the terms on either axis are nonnegative.  At the four points
$1,-1,i,-i$ the summands are exactly $1$.  Absolute convergence permits
summation, and
\[
  \sum_{a,b}e(a,b)
  =\frac14\left(\sum_{a\ne0}\frac1{a^2}\right)^2
  =\frac{\pi^4}{36}.
\]
It follows that $\operatorname{Re}G_4+\pi^4/36\geq4$.
The elementary bound $\pi^2<12$ makes the lower bound positive.
Reality of $g_2$ and $g_2=60G_4$ finish the proof.
The checked declarations are
\leanref{LemniscateSquareInvariantPositivity}{QuadraticPi.LemniscateSquareInvariantPositivity.G_four_re_lower_bound}
and
\leanref{LemniscateSquareInvariantPositivity}{QuadraticPi.LemniscateSquareInvariantPositivity.unit_g₂_re_pos}.
\end{proof}

Define
\begin{equation}\label{norm:rstar}
  r_0 = \sqrt{\sqrt{g_2(\Lambda_1)/4}}>0.
\end{equation}
Equations \eqref{norm:scaling} and \eqref{norm:rstar} imply
$g_2(\Lambda_{r_0})=4$ and $g_3(\Lambda_{r_0})=0$.
This construction has not used the integral to determine a lattice
invariant.  Its formal endpoints are
\leanref{LemniscateSquareNormalization}{QuadraticPi.LemniscateSquareNormalization.normalizingSide_positive},
\leanref{LemniscateSquareNormalization}{QuadraticPi.LemniscateSquareNormalization.normalizedPair_g₂},
and
\leanref{LemniscateSquareNormalization}{QuadraticPi.LemniscateSquareNormalization.normalizedPair_g₃}.

\subsection{Critical points and the real branch}

We first justify a critical-point fact without dividing by $\wp'$ at one
of its zeros.  For an arbitrary lattice $\Lambda$, differentiation of
the cubic gives
\[
  \wp'(z)\bigl(2\wp''(z)-12\wp(z)^2+g_2\bigr)=0
       \qquad(z\notin\Lambda).
\]
Both factors are analytic on the connected lattice complement.
The first is not identically zero, since $\wp'$ has a pole of order
three at the origin.  The analytic identity theorem therefore gives
\begin{equation}\label{norm:secondode}
  \wp''=6\wp^2-\frac{g_2}{2}
       \quad\hbox{on }\mathbb C\setminus\Lambda,
\end{equation}
including the zeros of $\wp'$.
See
\leanref{LemniscateDifferential}{QuadraticPi.LemniscateDifferential.order_derivWeierstrassP}
and
\leanref{LemniscateDifferential}{QuadraticPi.LemniscateDifferential.deriv_derivWeierstrassP}.

\begin{lemma}[Location of critical points]\label{norm:critical}
If $c\notin\Lambda$ and $\wp'(c)=0$, then $2c\in\Lambda$.
\end{lemma}
\begin{proof}
The analytic functions $f(z)=\wp(z)$ and $h(z)=\wp(2c-z)$ solve
the same equation \eqref{norm:secondode} near $c$, and have equal value
and first derivative there.  Analytic uniqueness is proved directly
from the Taylor recurrence.  For a solution $F''=aF^2-b$, with
$F(z_0+s)=\sum_{n\geq0}F_ns^n$, it reads
\[
  (n+2)(n+1)F_{n+2}
   =a\sum_{j=0}^nF_jF_{n-j}-b\,\mathbf1_{n=0}.
\]
Thus $F_0,F_1$ determine all coefficients.  Consequently $f=h$
near $c$.

The set obtained by deleting both $\Lambda$ and $2c-\Lambda$ from
$\mathbb C$ is connected, and the identity theorem extends $f=h$
throughout that set.  If $2c\notin\Lambda$, then $h$ is analytic at
$0$.  In a sufficiently small punctured neighborhood of $0$, the
extended identity would equate it with $\wp$, whose pole has order
two.  This contradiction proves the assertion.
The relevant exact results are
\leanref{LemniscateODEUniqueness}{QuadraticPi.LemniscateODEUniqueness.second_order_series_unique},
\leanref{LemniscateODEUniqueness}{QuadraticPi.LemniscateODEUniqueness.analytic_second_order_unique},
and
\leanref{LemniscateCriticalPoints}{QuadraticPi.LemniscateCriticalPoints.two_mul_mem_lattice_of_deriv_eq_zero}.
\end{proof}

\begin{proposition}[The normalized positive real branch]\label{norm:realbranch}
Let $r>0$ satisfy $g_2(\Lambda_r)=4$, and write
$X(t)=\wp_r(t)$, $Y(t)=\wp'_r(t)$ for real $t$ off the lattice.
Then
\begin{align}
  &X(t)>1,\quad Y(t)<0 &&(0<t<r/2),\label{norm:branchsigns}\\
  &X(r/2)=1,\quad Y(r/2)=0,\qquad
    \lim_{t\downarrow0}X(t)=+\infty.\label{norm:branchends}
\end{align}
Moreover $X'=Y$, $Y'=6X^2-2$, and $Y^2=4X^3-4X$ on the interval.
\end{proposition}
\begin{proof}
The real structure proved above supplies real-valued $X,Y$ and their
real differential equations.  The Laurent decomposition
$\wp_r(t)=t^{-2}+\text{a function analytic at }0$ gives the stated
limit.  No real point in $(0,r)$ belongs to $\Lambda_r$.
Lemma~\ref{norm:critical} consequently excludes a zero of $Y$ in
$(0,r/2)$, since such a zero would imply $2t\in\Lambda_r$ with
$0<2t<r$.  The cubic therefore excludes $X(t)=1$ on this interval.
Continuity, the pole limit, and the intermediate value theorem now
give $X(t)>1$ throughout it.

Oddness and $r$-periodicity of $\wp'_r$ imply $Y(r/2)=0$.
Continuity gives $X(r/2)\geq1$, and the cubic at $r/2$ restricts
$X(r/2)$ to $-1,0,1$, so its value is $1$.
Finally, $Y'=6X^2-2>0$ in the open interval.  Hence $Y$ is strictly
increasing towards its endpoint value $0$, which proves $Y<0$.
All functions are continuous at the right endpoint because it is not
a lattice point.

The individual checked conclusions are
\leanref{LemniscateRealBranch}{QuadraticPi.LemniscateRealBranch.real_weierstrassP_tendsto_atTop},
\leanref{LemniscateRealBranch}{QuadraticPi.LemniscateRealBranch.square_real_deriv_ne_zero},
\leanref{LemniscateRealBranch}{QuadraticPi.LemniscateRealBranch.one_lt_realP},
\leanref{LemniscateRealBranch}{QuadraticPi.LemniscateRealBranch.realP_halfPeriod_eq_one},
and
\leanref{LemniscateRealBranch}{QuadraticPi.LemniscateRealBranch.realDerivP_neg}.
The differential equations are supplied by
\leanref{LemniscateRealCalculus}{QuadraticPi.LemniscateRealCalculus.hasDerivAt_realP},
\leanref{LemniscateRealCalculus}{QuadraticPi.LemniscateRealCalculus.hasDerivAt_realDerivP_of_normalization},
and
\leanref{LemniscateRealCalculus}{QuadraticPi.LemniscateRealCalculus.real_cubic_of_normalization}.
\end{proof}

\begin{theorem}[Identification of the actual period]\label{norm:period}
If $r>0$ and $g_2(\Lambda_r)=4$, then $r=\Omega$.
In particular,
\begin{equation}\label{norm:invariants}
  g_2\bigl(\Omega(\mathbb Z+i\mathbb Z)\bigr)=4,
  \qquad
  g_3\bigl(\Omega(\mathbb Z+i\mathbb Z)\bigr)=0.
\end{equation}
\end{theorem}
\begin{proof}
Use the branch of Proposition~\ref{norm:realbranch}, and define
\[
  u(0)=0,\qquad u(t)=X(t)^{-1/2}\quad(0<t\leq r/2).
\]
The pole limit proves continuity at $0$.  The remaining branch facts
give continuity on $[0,r/2]$, the value $u(r/2)=1$, and $0<u(t)<1$
in its interior.  Direct differentiation gives
\[
  u'(t)=\frac{-Y(t)}{2X(t)^{3/2}}>0,
  \qquad
  u'(t)^2=\frac{4X(t)^3-4X(t)}{4X(t)^3}=1-u(t)^4.
\]
The sign $Y(t)<0$ therefore fixes the square root:
$u'(t)=\sqrt{1-u(t)^4}$.
This sign-sensitive step is
\leanref{LemniscateReciprocalBranch}{QuadraticPi.Lemniscate.reciprocal_sqrt_hasDerivAt}.

Let $F(v)=\int_0^v(1-s^4)^{-1/2}\,ds$.  Convergence of the integral
makes $F$ continuous on $[0,1]$, and the fundamental theorem of
calculus gives $F'(v)=(1-v^4)^{-1/2}$ on $(0,1)$.
Thus $(F\circ u)'=1$ on $(0,r/2)$, while $F\circ u$ is continuous
on the closed interval.  Integrating the constant derivative yields
\[
  \frac r2=F(1)-F(0)=\int_0^1\frac{ds}{\sqrt{1-s^4}},
\]
so $r=\Omega$.
The endpoint-safe integral argument is
\leanref{LemniscateInverseIntegral}{QuadraticPi.Lemniscate.halfPeriod_eq_integral_of_branch};
its real Weierstrass specialization is
\leanref{LemniscateReciprocalBranch}{QuadraticPi.Lemniscate.period_eq_Omega_of_real_branch}.
Apply this argument to $r_0$ from \eqref{norm:rstar} to obtain
$r_0=\Omega$ and \eqref{norm:invariants}.
The unconditional endpoint is
\leanref{LemniscateOmegaNormalization}{QuadraticPi.LemniscateOmegaNormalization.omegaPair_g₂_eq_four},
proved through
\leanref{LemniscateOmegaNormalization}{QuadraticPi.LemniscateOmegaNormalization.period_eq_Omega_of_normalization}
and
\leanref{LemniscateOmegaNormalization}{QuadraticPi.LemniscateOmegaNormalization.normalizingSide_eq_Omega}.
\end{proof}

\subsection{The diagonal half-period and the formal chart}

From now on let $\Lambda=\Omega(\mathbb Z+i\mathbb Z)$,
$\wp=\wp_\Lambda$, and
\begin{equation}\label{norm:zT}
  z_T=\frac{(1+i)\Omega}{2}.
\end{equation}
This point is not in $\Lambda$.
Because $iz_T+\Omega=z_T$ and $\wp(iz)=-\wp(z)$,
periodicity gives $\wp(z_T)=0$.  Oddness and periodicity similarly
give $\wp'(z_T)=0$.  These are exact values of the actual lattice
functions, proved in
\leanref{LemniscatePeriodGeometry}{QuadraticPi.LemniscatePeriod.omega_diagonalHalfPeriod_not_mem},
\leanref{LemniscatePeriodGeometry}{QuadraticPi.LemniscatePeriod.omega_weierstrassP_diagonalHalfPeriod},
and
\leanref{LemniscatePeriodGeometry}{QuadraticPi.LemniscatePeriod.omega_derivWeierstrassP_diagonalHalfPeriod}.
Notice that $z_T$ is the diagonal half-period; the real endpoint
$\Omega/2$ used above is a different half-period and has $\wp(\Omega/2)=1$.

Let $x_0(t)\in\mathbb Z[[t]]$ be the centered formal solution
\begin{equation}\label{norm:formalx}
  x_0(t)^3-x_0(t)=t^2,\qquad x_0(0)=0.
\end{equation}
It exists integrally because $X^3-X$ has invertible linear
coefficient $-1$: take its compositional inverse and substitute $t^2$.
Define $\ell(t)\in\mathbb Q[[t]]$ by
\begin{equation}\label{norm:formallog}
  \ell(0)=0,\qquad
  \ell'(t)\bigl(3x_0(t)^2-1\bigr)=1.
\end{equation}
The inverse of $3x_0^2-1$ has integral coefficients, so the coefficient
of $t^n$ in $\ell$ has denominator dividing $n$ for $n\geq1$.
In particular $\ell'(0)=-1$.
These are the actual series in the arithmetic construction, defined
by
\leanref{LemniscateLocal}{QuadraticPi.LemniscateLocal.localX}
and
\leanref{LemniscateLocal}{QuadraticPi.LemniscateLocal.ellipticLog},
with identities
\leanref{LemniscateLocal}{QuadraticPi.LemniscateLocal.localX_equation},
\leanref{LemniscateLocal}{QuadraticPi.LemniscateLocal.ellipticLog_derivative_equation},
and
\leanref{LemniscateLocal}{QuadraticPi.LemniscateLocal.ellipticLog_coeff_degree_cleared}.

\begin{proposition}[Identification with the actual germ]\label{norm:germ}
Let $A(s)$ and $B(s)$ be the Taylor series at $s=0$ of
\[
  \wp(z_T+s),\qquad \frac12\wp'(z_T+s),
\]
respectively.  After mapping the rational and integral coefficients
of $x_0,\ell$ into $\mathbb C$, one has
\begin{align}
  &A(0)=B(0)=0,\qquad A^3-A=B^2,\label{norm:germcubic}\\
  &A'=2B,\qquad B'=3A^2-1,\qquad [s]B=-1,\label{norm:germode}\\
  &x_0(B(s))=A(s),\qquad \ell(B(s))=s.\label{norm:germeval}
\end{align}
\end{proposition}
\begin{proof}
Analyticity at $z_T$, Theorem~\ref{norm:period}, and the actual
Weierstrass equations give \eqref{norm:germcubic}--\eqref{norm:germode}
under Taylor expansion.  The derivative identities are identities
of full formal power series.  In particular, no cancellation of
the zero scalar $B(0)$ is involved.

The centered solution of $x^3-x=B^2$ is unique.  Indeed, if $U,V$
are two centered solutions, then
\[
 (U-V)(U^2+UV+V^2-1)=0,
\]
and the second factor has constant coefficient $-1$, so is a unit.
This proves $x_0(B)=A$.  Formal differentiation, using
\eqref{norm:formallog}, now gives
\[
  \bigl(\ell(B)\bigr)'
    =\ell'(B)(3A^2-1)=1.
\]
Its constant coefficient is zero, proving $\ell(B)=s$.
The actual Taylor-series definitions are
\leanref{LemniscateTaylorBridge}{QuadraticPi.LemniscateTaylorBridge.omegaXSeries}
and
\leanref{LemniscateTaylorBridge}{QuadraticPi.LemniscateTaylorBridge.omegaYSeries}.
The unconditional specialization is supplied by
\leanref{LemniscateOmegaGerm}{QuadraticPi.LemniscateOmegaGerm.omega_series_cubic},
\leanref{LemniscateOmegaGerm}{QuadraticPi.LemniscateOmegaGerm.omegaYSeries_derivative},
\leanref{LemniscateOmegaGerm}{QuadraticPi.LemniscateOmegaGerm.omega_localX_subst_eq},
\leanref{LemniscateOmegaGerm}{QuadraticPi.LemniscateOmegaGerm.omega_ellipticLog_subst_eq},
and
\leanref{LemniscateOmegaGerm}{QuadraticPi.LemniscateOmegaGerm.omegaYSeries_linear}.
The identity $A'=2B$ is
\leanref{LemniscateTaylorBridge}{QuadraticPi.LemniscateTaylorBridge.omegaXSeries_derivative}.
\end{proof}

\subsection{The parameter used by the contour rows}

The analytic determinant uses the parameter $u$ with
$s=\log(1+u)$, where the principal complex logarithm is analytic
near $u=0$.  To identify its coefficients, define the actual functions
\begin{equation}\label{norm:logfunctions}
  a(u)=\wp\bigl(z_T+\log(1+u)\bigr),\qquad
  b(u)=\frac12\wp'\bigl(z_T+\log(1+u)\bigr).
\end{equation}
Write $a_*,b_*,L$ for their Taylor series at zero and for the Taylor
series of $\log(1+u)$, respectively.  They satisfy
\begin{equation}\label{norm:loggerm}
  a_*^3-a_*=b_*^2,\qquad
  b_*'=(3a_*^2-1)L',\qquad
  x_0(b_*)=a_*,\qquad \ell(b_*)=L.
\end{equation}
All constant coefficients vanish, and
\begin{equation}\label{norm:loglinear}
  [u]L=1,\qquad [u]b_*=-1.
\end{equation}
For completeness, the second identity in \eqref{norm:loggerm}
is the chain rule applied to $\wp'/2$, whose derivative is
$3\wp^2-1$ after normalization.  Evaluating its constant coefficient
gives $[u]b_*=-1$.  Centered uniqueness gives $x_0(b_*)=a_*$,
and differentiating $\ell(b_*)$ gives $L'$; equality of their
constant coefficients then gives the last identity.
The logarithm derivative is proved for $|u|<1$, and the elliptic
compositions are analytic on a neighborhood of zero because $z_T$
is not a lattice point.  Only their local Taylor identities are
asserted here; the later contour estimates use the pole-cleared
entire functions.

The exact function and series definitions are in
\leanref{LemniscateLogParameterChart}{QuadraticPi.LemniscateLogParameterChart.logParameter},
\leanref{LemniscateLogParameterChart}{QuadraticPi.LemniscateLogParameterChart.xFunction},
\leanref{LemniscateLogParameterChart}{QuadraticPi.LemniscateLogParameterChart.yFunction},
\leanref{LemniscateLogParameterChart}{QuadraticPi.LemniscateLogParameterChart.xSeries},
\leanref{LemniscateLogParameterChart}{QuadraticPi.LemniscateLogParameterChart.ySeries},
and
\leanref{LemniscateLogParameterChart}{QuadraticPi.LemniscateLogParameterChart.ellSeries}.
The unconditional identities used in the proof are
\leanref{LemniscateOmegaGerm}{QuadraticPi.LemniscateOmegaGerm.log_series_cubic},
\leanref{LemniscateOmegaGerm}{QuadraticPi.LemniscateOmegaGerm.log_ySeries_derivative},
\leanref{LemniscateOmegaGerm}{QuadraticPi.LemniscateOmegaGerm.log_localX_subst},
\leanref{LemniscateOmegaGerm}{QuadraticPi.LemniscateOmegaGerm.log_ellipticLog_subst},
and
\leanref{LemniscateOmegaGerm}{QuadraticPi.LemniscateOmegaGerm.log_ySeries_linear}.
Together with
\leanref{LemniscateLogParameterChart}{QuadraticPi.LemniscateLogParameterChart.ellSeries_linear},
they give the actual coefficient substitution used to transport the
arithmetic jet matrix.  In particular its base-parameter linear
coefficient is exactly $-1$, of absolute value one.

\section{Elliptic separated-weight interpolation}
\label{epa:section}

This section states the interpolation theorem with its full quantifier order
and records the geometric argument used in the formal proof.  Throughout,
\[
 E:\quad y^2=x^3-x,
 \qquad
 R_m=\mathbb C[x,y,X_1,\ldots,X_m]/(y^2-x^3+x).
\]
We use the unique normal representatives
\(\sum a_{a,b,\alpha}x^a y^b X^\alpha\), where \(b\in\{0,1\}\).
The elliptic source degree is the pole degree \(2a+3b\).  In particular,
\(x\) and \(y\) are not two independent directions of the source filtration.

\subsection{The formal chart and the exact statement}
\label{epa:statement-subsection}

Let \(x_*(t)\in\mathbb Z[[t]]\) be the unique series with constant term zero
satisfying
\begin{equation}\label{epa:formal-cubic}
 x_*(t)^3-x_*(t)=t^2.
\end{equation}
It is obtained by composing the integral compositional inverse of
\(X^3-X\) with \(t^2\); the coefficient of \(X\) in \(X^3-X\) is the unit
\(-1\).  Set
\begin{equation}\label{epa:formal-log}
 J(t)=3x_*(t)^2-1,
 \qquad V(t)=J(t)^{-1}\in\mathbb Z[[t]],
 \qquad
 \ell(t)=\sum_{n\geq1}\frac{[t^{n-1}]V(t)}{n}\,t^n.
\end{equation}
Thus
\begin{equation}\label{epa:formal-differential}
 \ell(0)=0,\qquad \ell'(0)=-1,\qquad
 (3x_*(t)^2-1)\ell'(t)=1.
\end{equation}
These are the actual formal series in the code, rather than unspecified
solutions of a formal differential equation.
See \leanref{LemniscateLocal}{QuadraticPi.LemniscateLocal.localX_equation}
and
\leanref{LemniscateLocal}{QuadraticPi.LemniscateLocal.ellipticLog_derivative_equation}.

Fix integers \(m,K\geq1\) and rational numbers \(w_0,v_0,\theta\) such that
\begin{equation}\label{epa:base-data}
 w_0>0,\quad v_0>0,\quad 0<\theta<1,\quad
 K\theta^m<1,\quad
 K\frac{w_0}{v_0}\theta^m<1.
\end{equation}
These are precisely the fields and inequalities of
\leanref{SeparatedWeightPA}{QuadraticPi.SeparatedWeightPA.BaseData}.
For positive rational additive weights \(w=(w_1,\ldots,w_m)\), put
\begin{align}
 \mathcal P_H(w)
 &=\operatorname{span}_{\mathbb C}
 \left\{x^a y^bX^\alpha:
 b\in\{0,1\},\quad
 w_0(2a+3b)+\sum_{i=1}^m w_i\alpha_i\leq H\right\},
 \label{epa:source-space}\\
 \mathcal B_H(w)
 &=\left\{(k,\beta)\in\mathbb N\times\mathbb N^m:
 v_0k+\sum_{i=1}^m\frac{w_i}{\theta}\beta_i<H\right\}.
 \label{epa:target-packet}
\end{align}
The source cutoff is weak and the target cutoff is strict.  For centers
\(c_j=(c_{j1},\ldots,c_{jm})\in\mathbb C^m\), define the linear packet map
\begin{equation}\label{epa:packet-map}
 \begin{split}
 \mathcal E_{H,w,c}:\mathcal P_H(w)&\longrightarrow
       \prod_{j=1}^K\mathbb C^{\mathcal B_H(w)},\\
 P&\longmapsto
 \left([t^ku^\beta]
 P\bigl(x_*(t),t,(c_{ji}+u_i+\ell(t))_{i=1}^m\bigr)
 \right)_{j,(k,\beta)}.
 \end{split}
\end{equation}
The cubic relation is respected by this substitution, as proved by
\leanref{LemniscateInterpolation}{QuadraticPi.LemniscateInterpolation.formalJet_cubic_relation}.

\begin{theorem}[Elliptic separated-weight interpolation, EPA]
\label{epa:theorem}
For every choice of data satisfying \eqref{epa:base-data}, there exist
functions
\[
 T_i:\mathbb Q^{i-1}\longrightarrow\mathbb Q
 \qquad(1\leq i\leq m)
\]
with the following property.  Whenever the positive rational weights satisfy
\[
 w_i>T_i(w_1,\ldots,w_{i-1})\qquad(1\leq i\leq m),
\]
there is an integer \(D>0\) such that, for every center configuration satisfying
\[
 c_{ji}\neq c_{j'i}\qquad
 (j\neq j',\ 1\leq i\leq m),
\]
there is an integer \(H_0\) for which \(\mathcal E_{H,w,c}\) is surjective
whenever \(H\geq H_0\) and \(D\mid H\).
\end{theorem}

The order of choices is
\begin{equation}\label{epa:quantifiers}
 (m,K,w_0,v_0,\theta)\ ;\ T\ ;\ w\ ;\ D\ ;\ c\ ;\ H_0\ ;\ H.
\end{equation}
Thus each threshold sees only earlier weights, and the divisibility modulus
is independent of the centers.  The size threshold \(H_0\) may depend on the
centers.  Coordinatewise distinctness does not impose distinctness between
different coordinates.  The exact proposition is
\leanref{LemniscateInterpolation}{QuadraticPi.LemniscateInterpolation.Statement};
its unconditional proof is
\leanref{LemniscateGlobalInterpolationTheorem}{QuadraticPi.LemniscateGlobalInterpolation.ellipticSeparatedWeightPA}.

\subsection{The curve inequality underlying EPA}
\label{epa:curve-subsection}

We give the precise geometric intermediate statement.  Write
\[
 W=(w_0,w_1,\ldots,w_m),\quad
 V=(v_0,w_1/\theta,\ldots,w_m/\theta),\quad
 \rho=(2w_0,3w_0,w_1,\ldots,w_m).
\]
Consider a one-variable function field \(F/\mathbb C\) generated by a tuple
\(z=(x,y,X_1,\ldots,X_m)\) satisfying the cubic equation.  At a normalized
place \(P\), let \(\operatorname{ord}_P\) be its valuation.  The weighted
pole degree is
\begin{equation}\label{epa:weighted-degree}
 d_\rho(z)=\sum_P
 \max\left(0,\frac{-\operatorname{ord}_P(x)}{2w_0},
                 \frac{-\operatorname{ord}_P(y)}{3w_0},
                 \max_i\frac{-\operatorname{ord}_P(X_i)}{w_i}\right).
\end{equation}
Zero coordinates contribute zero to the pole maximum.  All but finitely
many terms vanish.  At a place centered at \((0,0,c_j)\), expand the tuple in
the completed local ring and put
\begin{equation}\label{epa:contact}
 \tau_P=\min\left\{
 \frac{\operatorname{ord}_P(y)}{v_0},\quad
 \frac{\operatorname{ord}_P(X_i-c_{ji}-\ell(y))}{w_i/\theta}:
 1\leq i\leq m\right\},
\end{equation}
where zero series are omitted from the minimum.  At least one of the series
is nonzero on a nonconstant curve.  This is the actual completed-local-ring
contact used by
\leanref{LemniscatePlaceContact}{QuadraticPi.LemniscatePlaceContact.contact}.
In particular \(\ell(y)\) in \eqref{epa:contact} is a formal substitution in
the completed local ring, not an assumed rational function on the curve.

\begin{proposition}[Uniform weighted curve inequality]
\label{epa:curve-inequality}
The successive thresholds in Theorem~\ref{epa:theorem} can be chosen so that
there is a rational \(\sigma>0\), chosen before the centers, for which every
such function-field curve and every allowed center configuration satisfy
\begin{equation}\label{epa:curve-margin}
 (1+\sigma)\sum_{P\text{ centered at some }(0,0,c_j)}\tau_P
 \leq d_\rho(z).
\end{equation}
\end{proposition}

\begin{proof}
First choose \(\sigma>0\) sufficiently small that
\begin{equation}\label{epa:inflated-volumes}
 K(1+3\sigma)^{m+1}\frac{w_0}{v_0}\theta^m<1,
 \qquad K(1+3\sigma)^m\theta^m<1,
 \qquad (1+\sigma)\theta<1.
\end{equation}
Set
\begin{equation}\label{epa:epsilon}
 \epsilon=\min\left(\sigma,
       \frac{\sigma v_0}{2(1+\sigma)w_0}\right),
 \qquad g=1+\frac{w_0}{v_0}\epsilon.
\end{equation}
Then \(\epsilon>0\), \(g<3/2\), and
\begin{equation}\label{epa:growth-margin}
 g(1+\sigma)<1+3\sigma-\epsilon.
\end{equation}
These elementary estimates are checked in
\leanref{LemniscateGrowthBudget}{QuadraticPi.LemniscateGrowthBudget.strict_growth_contact_margin}
and
\leanref{LemniscateGrowthBudget}{QuadraticPi.LemniscateGrowthBudget.growth_factor_lt_three_halves}.

For clarity, the comparison constant used in the formal construction is
\begin{align}
 C_r(s)&=2(r+2)^{r+2}\left(1+\frac{r+2}{s}\right)^{r+2},
 \label{epa:comparison-basic}\\
 C_B(\sigma)&=C_m(\sigma)+3\,2^{m+3}C_{m+1}(\epsilon).
 \label{epa:comparison-constant}
\end{align}
Successive growth of the \(w_i\) gives the following product separation:
if \(A,A'\subset\{0,\ldots,m\}\) have equal cardinality and their largest
differing index \(i>0\) belongs to \(A\), then
\begin{equation}\label{epa:product-separation}
 C_B(\sigma)\prod_{j\in A'}V_j<\prod_{j\in A}W_j.
\end{equation}
All thresholds needed for \eqref{epa:inflated-volumes} and
\eqref{epa:product-separation} precede the centers.
This is
\leanref{LemniscateComparisonConstant}{QuadraticPi.LemniscateComparisonConstant.exists_thresholds},
using the explicit numerical record
\leanref{LemniscateWeightGeometry}{QuadraticPi.LemniscateWeightGeometry.NumericData}.

The elliptic differential frame is
\begin{equation}\label{epa:frame}
 \mathcal D_0=2y\frac{\partial}{\partial x}
 +(3x^2-1)\frac{\partial}{\partial y}
 +\sum_{i=1}^m\frac{\partial}{\partial X_i},
 \qquad \mathcal D_i=\frac{\partial}{\partial X_i}\ (i\geq1).
\end{equation}
It preserves the cubic ideal.  Under the formal chart it becomes
\(J(t)\partial_t,\partial_{u_1},\ldots,\partial_{u_m}\).
The source-weight cost of \(\mathcal D_0\) is at most \(w_0\); that of
\(\mathcal D_i\) is \(-w_i\).  Its loss of target vanishing is respectively
\(v_0\) and \(w_i/\theta\).  These signed source costs and target losses are
proved for every derivative word, rather than imposed as hypotheses; see
\leanref{LemniscateGeometryFrame}{QuadraticPi.LemniscateGeometryFrame.formalJet_polynomialFrame}
and
\leanref{LemniscateWeightedDerivation}{QuadraticPi.LemniscateWeightedDerivation.polynomialFrame_supportBound}.

Suppose \eqref{epa:curve-margin} fails.  The first strict volume inequality
in \eqref{epa:inflated-volumes} yields, for a sufficiently large integer
\(n\), a normal polynomial \(P\) of source weight at most \(n\), nonzero
modulo the cubic ideal, whose packets vanish below
\((1+3\sigma)n\) at all centers.  The dimension argument uses the injective
monomial map
\[
 t^sX^\alpha\longmapsto x^{\lfloor s/2\rfloor}y^{s\bmod2}X^\alpha,
\]
whose source-degree loss is at most \(2w_0\).  Strictness of the volume
inequality absorbs this fixed loss.
See
\leanref{LemniscateAuxiliaryPolynomial}{QuadraticPi.LemniscateAuxiliaryPolynomial.eventually_exists_nonzero_elliptic_auxiliary}.
For a derivative word of target cost at most \(\epsilon n\), its source
weight is at most \(gn\), and its local vanishing order is at least
\((1+3\sigma-\epsilon)n\tau_P\).  The degree-zero principal-divisor formula,
the assumed excess of contact, and \eqref{epa:growth-margin} force each of
these derivatives to vanish identically on the curve.  This supplies one
actual auxiliary polynomial with all the persistence hypotheses:
\leanref{LemniscateCurveAuxiliary}{QuadraticPi.LemniscateCurveAuxiliary.exists_auxiliary_with_persistence_inputs}.

Apply the persistent-prime argument to the ideals generated by the cubic
and the bounded derivative words of this same \(P\).  At a persistent
prime of intrinsic codimension \(1\leq r\leq m+1\), the upper weighted degree
estimate and the lower rectangular jet estimate apply to the same
localized ideal and the same length \(L\).  In particular their relevant
bounds have the form
\begin{equation}\label{epa:local-length-comparison}
 L\prod_{i\in A}W_i\leq3(gn)^r,
 \qquad
 \prod_{j=1}^r R_j\leq L,
\end{equation}
where, if \(b_1,\ldots,b_r\) enumerate the frame-normal basis \(A'\),
\[
 R_j=\left\lceil\frac{\epsilon n}{(m+3)rV_{b_j}}\right\rceil.
\]
These are the rectangular jet cutoffs in the independent frame-normal
directions for derivative budget \(\epsilon n/(m+3)\).
Taking \(n\) large ensures each cutoff is at least two.  The comparison
constant \eqref{epa:comparison-constant} absorbs the rounding margins, the
factor three from the fixed cubic, and all \(m+3\) persistence stages.
The resulting inequality is
\[
 \prod_{i\in A}W_i\leq C_B(\sigma)\prod_{j\in A'}V_j
\]
for the corresponding coordinate-normal and frame-normal bases.
This comparison is proved in
\leanref{LemniscateNormalComparison}{QuadraticPi.LemniscateNormalComparison.normal_product_comparison};
its two local-length inputs are in \texttt{LemniscatePrimeUpper.lean} and
\texttt{LemniscatePrimeLower.lean}.

Combining this comparison with \eqref{epa:product-separation} excludes
every prohibited change of normal basis.  The resulting differential
rigidity forces the elliptic projection of the persistent component to
be constant.  The obstruction used here is concrete: the invariant
differential
\[
 \frac{dx}{2y}=\frac{dy}{3x^2-1}
\]
has no primitive on any finite cover of the nonconstant elliptic
function field. This is proved over any characteristic-zero constant field
$K$. Put $h=x^3-x$, let $n=[E:K(x)]$, and suppose that a primitive $z$
exists in the finite extension $E$. For the derivation with $x'=1$, the
identity $2yz'=1$ and the twisted trace $T=\operatorname{Tr}_{E/K(x)}(yz)$
give
\[
 2hT'=h'T+nh.
\]
Thus $b=T/(nh)\in K(x)$ satisfies the rational differential equation
\[
 2(x^3-x)b'(x)+(3x^2-1)b(x)=1.
\]
Finite poles are excluded by their orders (the roots of \(x^3-x\) are
simple), and a nonzero polynomial solution is excluded by its leading
term.  The actual finite-cover and projection statements are
\leanref{LemniscateDifferentialRigidity}{QuadraticPi.LemniscateDifferentialRigidity.finite_extension_no_exact_cubicDifferential}
and
\leanref{LemniscateCurveRigidity}{QuadraticPi.LemniscateCurveRigidity.constant_projection}.
Because the original curve passes through an interpolation center, its
constant elliptic coordinates must be \((0,0)\).  The remaining curve is
therefore in the additive fiber; its contact and weighted degree agree
with those in the proved additive-fiber curve inequality.  The second
volume inequality in \eqref{epa:inflated-volumes} and the ordinary part
of \(C_B(\sigma)\) supply that inequality and contradict the supposed
excess.  The complete assembly is
\leanref{LemniscateCurveInequality}{QuadraticPi.LemniscateCurveInequality.weighted_curve_inequality}.
\end{proof}

\subsection{From the curve inequality to the packet map}
\label{epa:geometric-assembly}

\begin{proof}[Proof of Theorem~\ref{epa:theorem}]
Choose a common positive integer radius \(D\) for the rational source and
target weights.  Thus all numbers
\[
 \frac{D}{2w_0},\ \frac{D}{3w_0},\ \frac{D}{w_i},\quad
 a=\frac{D}{v_0},\quad b_i=\frac{D}{w_i/\theta}
\]
are positive integers.  Choose a single integer \(T>0\) with
\(Tv_0>w_i/\theta\) for every \(i\).  These choices are made before the
centers, as in
\leanref{LemniscateFixedJetScale}{QuadraticPi.LemniscateFixedJetScale.exists_scale}.
Let \(\ell_{<T}\) denote the truncation of \(\ell\) in degrees less than
\(T\), and define in the affine coordinate ring
\begin{equation}\label{epa:fixed-center-ideal}
 I_j=\left(y^a,x^a,
       (X_i-c_{ji}-\ell_{<T}(y))^{b_i}\ (1\leq i\leq m)\right).
\end{equation}
The generator \(x^a\) selects the branch \((x,y)=(0,0)\) from the three
points above \(y=0\).  Locally at that branch it is redundant, since
\(x(x^2-1)=y^2\) and \(x^2-1\) is a unit.  Globally it is essential to
the support statement.  These facts are proved in
\leanref{LemniscateJetSelector}{QuadraticPi.LemniscateJetSelector.localized_selector_eq}
and
\leanref{LemniscateJetSelector}{QuadraticPi.LemniscateJetSelector.selectorIdeal_le_prime_iff}.
The fixed ideal \(I=\prod_j I_j\) is extended to a coherent ideal on the
projective monomial compactification at radius \(D\); its support consists
exactly of the chosen centers.

Blow up this ideal.  Denote by \(A\) the pullback hyperplane bundle and by
\(J\) the invertible exceptional ideal.  For every noncontracted integral
curve meeting the affine chart, use its own normalized function field.
The two exact degree formulas are
\begin{equation}\label{epa:degree-formulas}
 \deg_C A=D\,d_\rho(z),\qquad
 \deg_C J=-D\sum_P\tau_P.
\end{equation}
The second formula follows from the local colength of
\eqref{epa:fixed-center-ideal}; the strict cutoff \(Tv_0>w_i/\theta\)
ensures that replacing \(\ell\) by \(\ell_{<T}\) does not change the
weighted minimum.  The formulas, including construction of the required
normalization data, are
\leanref{LemniscateCurveHyperplaneDegree}{QuadraticPi.LemniscateCurveModel.hyperplane_degree}
and
\leanref{LemniscateCurveExceptionalDegree}{QuadraticPi.LemniscateCurveModel.exceptional_degree}.

An auxiliary twist \(A^{\otimes a_0}\otimes J\) is ample for some
\(a_0>1\).  This auxiliary exponent may depend on the centers.  It does not
alter \(D\), the ideal, or the desired interpolation bundle \(A\otimes J\).
Equation~\eqref{epa:curve-margin}, together with the two degree formulas,
gives the uniform comparison
\begin{equation}\label{epa:uniform-ample-margin}
 \deg_C(A\otimes J)\geq
 \frac{\sigma}{a_0(1+\sigma)-1}
       \deg_C(A^{\otimes a_0}\otimes J).
\end{equation}
Contracted curves and curves outside the affine chart satisfy the same
comparison by the exceptional-bundle and support calculations.  The
coefficient is positive and independent of the curve.  The proved
numerical ampleness theorem therefore gives ampleness of \(A\otimes J\).
This is
\leanref{LemniscateBlowupAmple}{QuadraticPi.LemniscateBlowupGeometry.interpolationBundle_ample},
transported to the literal jet-surjectivity bundle by
\leanref{LemniscateGlobalInterpolationTheorem}{QuadraticPi.LemniscateGlobalInterpolation.actual_blowupBundle_ample}.

The Rees-algebra and cohomological jet-surjectivity theorem now gives,
for all sufficiently large \(n\), surjectivity of the restriction of
degree-\(n\) global sections to the quotient by \(I^n\).
On the affine chart these sections have normal polynomial
representatives of source weight at most \(nD\).  Moreover each
\(I_j^n\) maps into the weighted formal ideal of cutoff \(nD\).
The ideals at distinct centers are comaximal, so their finite quotient
maps give precisely the product of coefficient packets in
\eqref{epa:packet-map}.  The section-to-normal-polynomial and
quotient-to-packet transports are included in
\leanref{LemniscateGlobalInterpolation}{QuadraticPi.LemniscateGlobalInterpolation.eventual_packetMap_of_ample}.
Taking \(H=nD\) proves the theorem.  Its exact divisibility formulation is
\leanref{LemniscateGlobalInterpolation}{QuadraticPi.LemniscateGlobalInterpolation.eventual_divisible_packetMap_of_ample}.
\end{proof}

\section{The actual arithmetic minor and its lower bound}
\label{epa:arithmetic-section}

\subsection{Gaussian centers and finite coefficient entries}

In the determinant application, take \(K=N^2\) and enumerate the Gaussian
square
\[
 \lambda_j\in\{a+bi:0\leq a,b<N\}.
\]
For integers \(p_i\neq0\) and \(q_i\geq2\), set
\begin{equation}\label{epa:rational-data}
 w_i=\lceil\log q_i\rceil,\qquad
 c_{ji}=\lambda_j\frac{p_i}{q_i}.
\end{equation}
The weights are positive integers, hence rational, and the centers are
coordinatewise distinct.  Fix \(F\geq1/\theta\) and put
\begin{equation}\label{epa:truncations}
 U_i=\left\lceil\frac{Fw_i}{v_0}\right\rceil,\qquad
 T_i=U_i+1,\qquad
 S=\left\lceil\frac{H}{v_0}\right\rceil.
\end{equation}
Let \(x_{<S}(t)\) and \(\ell_{<T_i}(t)\) be the corresponding finite
truncations of the series in \eqref{epa:formal-cubic}--\eqref{epa:formal-log}.
Rows are indexed by
\[
 \mathcal R_H=\{(j,k,\beta):1\leq j\leq N^2,\ (k,\beta)\in\mathcal B_H(w)\},
 \qquad M_H=|\mathcal R_H|,
\]
and columns by the normal monomials in \eqref{epa:source-space}.  Define
the rectangular matrix \(A_H\) by
\begin{equation}\label{epa:entry}
 (A_H)_{(j,k,\beta),(a,b,\alpha)}
 =\left(\prod_{i=1}^m\binom{\alpha_i}{\beta_i}\right)
 [t^k]\left(
 x_{<S}(t)^a t^b
 \prod_{i=1}^m
       (c_{ji}+\ell_{<T_i}(t))^{\alpha_i-\beta_i}\right).
\end{equation}
If any \(\beta_i>\alpha_i\), the entry is zero.  Formula
\eqref{epa:entry} is the proved coefficient extraction of the actual
evaluated monomial, not a separately specified arithmetic model; see
\leanref{LemniscateArithmetic}{QuadraticPi.LemniscateArithmetic.coefficient_monomialImage}.

The base-coordinate truncation does not affect any row, because
\(Sv_0\geq H\).  For the additive logarithm truncations, the tails begin
in degrees \(T_i\) and satisfy \(T_iv_0\geq w_i/\theta\).
Consequently the substitution
\[
 u_i\longmapsto u_i+\ell(t)-\ell_{<T_i}(t)
\]
and its inverse preserve the strict weighted packet quotient.
This is an invertible filtered change of target coordinates.  It proves
equivalence of the full and truncated packet surjectivity assertions;
it is not a claim of entrywise equality of the two matrices.
The formal bridge is
\leanref{LemniscateInterpolationMatrix}{QuadraticPi.LemniscateInterpolationMatrix.formal_packets_surjective_iff_truncated},
with the literal matrix application
\leanref{LemniscateInterpolationMatrix}{QuadraticPi.LemniscateInterpolationMatrix.matrix_surjective_of_epa_packet}.

\begin{proposition}[Cofinal nonzero selected minors]
\label{epa:cofinal-minor}
For fixed base data with \(K=N^2\), there are successive thresholds such
that, whenever \eqref{epa:rational-data} satisfies those thresholds,
for every \(F\geq1/\theta\) and every integer \(L\) there is an integer
\(H\geq L\), \(H>0\), and an injective column selection
\(\iota:\mathcal R_H\longrightarrow\mathcal P_H\)'s monomial basis such that
\begin{equation}\label{epa:selected-determinant}
 \Delta_H=\det\bigl((A_H)_{r,\iota(s)}\bigr)_{r,s\in\mathcal R_H}\neq0.
\end{equation}
The truncation \(S\) may be chosen as in \eqref{epa:truncations}.
\end{proposition}

\begin{proof}
Apply EPA at the rational Gaussian centers.  Its positive divisibility
progression contains arbitrarily large \(H\).  The filtered coordinate
change above transports packet surjectivity to surjectivity of the
linear map defined by \(A_H\).  Finite-dimensional linear algebra then
selects \(M_H\) independent columns.  The formal endpoint is
\leanref{LemniscateUnconditionalMinors}{QuadraticPi.LemniscateUnconditionalMinors.cofinal_minors}.
\end{proof}

\subsection{Denominator clearing without an elliptic source loss}

Let \(\Lambda(n)=\operatorname{lcm}(1,\ldots,n)\), with
\(\Lambda(0)=1\), and define
\begin{equation}\label{epa:denominator}
 D_H=\prod_{i=1}^m\Lambda(U_i)^{\lfloor H/w_i\rfloor},\qquad
 r_\beta=\prod_iq_i^{-\beta_i},\qquad
 c_\alpha=\prod_iq_i^{\alpha_i}.
\end{equation}
The coefficients of \(x_{<S}\) are integral.
For \(1\leq n<T_i\), equation~\eqref{epa:formal-log} gives
\[
 n[t^n]\ell(t)=[t^{n-1}]V(t)\in\mathbb Z.
\]
Hence \(\Lambda(U_i)\ell_{<T_i}\in\mathbb Z[t]\), and
\(q_i\Lambda(U_i)(c_{ji}+\ell_{<T_i})\in\mathbb Z[i][t]\).
Since \(\alpha_i\leq\lfloor H/w_i\rfloor\), every selected matrix entry
satisfies
\begin{equation}\label{epa:entry-integrality}
 D_H r_\beta c_\alpha
      (A_H)_{(j,k,\beta),(a,b,\alpha)}\in\mathbb Z[i].
\end{equation}
Only the remaining additive degree \(\alpha_i-\beta_i\) needs clearing.
There is no extra denominator contribution from the elliptic basis
\(x^ay^b\).  These assertions are
\leanref{LemniscateLocal}{QuadraticPi.LemniscateLocal.ellipticLog_trunc_lcm_cleared}
and
\leanref{LemniscateDeterminant}{QuadraticPi.LemniscateDeterminant.selectedMinor_entries_gaussian}.

Because a nonzero Gaussian integer has absolute value at least one,
\eqref{epa:entry-integrality} and \(\Delta_H\neq0\) give
\begin{equation}\label{epa:exact-clearing-bound}
 \log|\Delta_H|\geq
 -M_H\log D_H
 -\sum_{s\in\mathcal R_H}\sum_i\alpha_i(\iota(s))\log q_i
 +\sum_{(j,k,\beta)\in\mathcal R_H}\sum_i\beta_i\log q_i.
\end{equation}
This exact row/column cost identity is formalized by
\leanref{LemniscateDeterminant}{QuadraticPi.LemniscateDeterminant.selectedMinor_clearing_bound}.

Define the normalized mean additive row weight
\begin{equation}\label{epa:mean-row-weight}
 \overline W_H=
 \frac{1}{M_H H}
 \sum_{(j,k,\beta)\in\mathcal R_H}\sum_{i=1}^m w_i\beta_i,
 \qquad c_{\mathrm{lcm}}=\log4+4.
\end{equation}
For any \(0<w_{\min}\leq\min_iw_i\), the estimates used in
\eqref{epa:exact-clearing-bound} are
\begin{align}
 \log D_H
 &\leq H\left(c_{\mathrm{lcm}}\frac{Fm}{v_0}
                +c_{\mathrm{lcm}}\sum_i\frac1{w_i}\right),
 \label{epa:denominator-log}\\
 \sum_{s\in\mathcal R_H}\sum_i\alpha_i(\iota(s))\log q_i
 &\leq M_H H,
 \label{epa:column-cost}\\
 \sum_{(j,k,\beta)\in\mathcal R_H}\sum_i\beta_i\log q_i
 &\geq M_H H\left(\overline W_H-\frac{\theta}{w_{\min}}\right).
 \label{epa:row-cost}
\end{align}
For \eqref{epa:denominator-log}, use
\(\log\Lambda(n)\leq c_{\mathrm{lcm}}n\),
\(U_i\leq Fw_i/v_0+1\), and
\(\lfloor H/w_i\rfloor\leq H/w_i\).
For \eqref{epa:column-cost}, use \(\log q_i\leq w_i\) and the source
budget.  Finally \(w_i-\log q_i\leq1\), while each row has
\(\sum_iw_i\beta_i<\theta H\); this gives
\(\sum_i\beta_i\leq\theta H/w_{\min}\) and proves
\eqref{epa:row-cost}.

\begin{theorem}[Normalized arithmetic lower bound]
\label{epa:arithmetic-lower-bound}
Every nonzero selected minor above satisfies
\begin{equation}\label{epa:normalized-lower}
 \boxed{\quad
 \frac{\log|\Delta_H|}{M_H H}
 \geq -(1-\overline W_H)
 -\left(c_{\mathrm{lcm}}\frac{Fm}{v_0}
       +c_{\mathrm{lcm}}\sum_{i=1}^m\frac1{w_i}
       +\frac{\theta}{w_{\min}}\right).
 \quad}
\end{equation}
For data above the EPA thresholds, minors satisfying
\eqref{epa:normalized-lower} exist for arbitrarily large positive integers
\(H\).
\end{theorem}

\begin{proof}
Substitute \eqref{epa:denominator-log}--\eqref{epa:row-cost} into
\eqref{epa:exact-clearing-bound} and divide by \(M_HH>0\).
The lower bound for an arbitrary nonzero literal selected minor is
\leanref{LemniscateDeterminant}{QuadraticPi.LemniscateDeterminant.selectedMinor_arithmetic_lower_bound}.
Applying Proposition~\ref{epa:cofinal-minor} gives the unconditional cofinal
statement
\leanref{LemniscateUnconditionalMinors}{QuadraticPi.LemniscateUnconditionalMinors.cofinal_arithmetic_lower_bounds}.
This is the exact minor and exact lower estimate used in the analytic
comparison later in the proof.
\end{proof}

\section{An analytic bound for the same selected minor}
\label{ana:sec:analytic}

The arithmetic lower bound and the analytic upper bound must concern the
\emph{same} determinant.  We give the identification explicitly, including
the common pole multiplier, the local change of variable, and the finite
translation of the additive variables.  In particular, none of the analytic
steps makes a new choice of source columns.

\subsection{Rows, columns, and the actual arithmetic determinant}
\label{ana:subsec:notation}

Fix an integer $m\geq1$, an integer $N\geq1$, and put $K=N^2$.
Let $w_0,v_0,\theta$ be positive rational numbers with $\theta<1$.
For integers $p_i\ne0$ and $q_i\geq2$, $1\leq i\leq m$, set
\[
 w_i=\lceil\log q_i\rceil,\qquad
 W(\alpha)=\sum_{i=1}^m w_i\alpha_i.
\]
The row packet and the full row set at a positive cutoff $H$ are
\begin{align*}
 \mathcal B_H&=\{(k,\beta)\in\mathbb N\times\mathbb N^m:
                    v_0k+\theta^{-1}W(\beta)<H\},\\
 \mathcal R_H&=\{0,\ldots,N-1\}^2\times\mathcal B_H,
 &M_H&=|\mathcal R_H|=K|\mathcal B_H|.
\end{align*}
The source columns are the canonical elliptic monomials
\[
 \mathcal C_H=\{(a,b,\alpha):a\in\mathbb N,
       \ b\in\{0,1\},\ \alpha\in\mathbb N^m,
       \ w_0(2a+3b)+W(\alpha)\leq H\}.
\]
For a row $r=(j,k,\beta)$ write $k_r=k$, $\beta_r=\beta$, and
$\lambda_j=j_1+\mathrm i j_2$.  Define the normalized mean transverse
row weight by
\begin{equation}
 \overline W_H=\frac{1}{M_HH}\sum_{r\in\mathcal R_H}W(\beta_r).
 \label{ana:eq:mean}
\end{equation}
The zero jet belongs to every packet, and consequently $M_H>0$.
The defining row inequality gives
\begin{equation}
 0\leq\overline W_H\leq\theta,
 \qquad \sum_{r\in\mathcal R_H}k_r\leq M_HH/v_0.
 \label{ana:eq:rowbounds}
\end{equation}
These are the concrete rows used in
\leanref{LemniscateAnalyticRowData}{QuadraticPi.LemniscateAnalyticRowData.row_order_sum_le}
and
\leanref{LemniscateExponentConditional}{QuadraticPi.LemniscateExponentConditional.mean_bounds}.

For an integer truncation level $L$ satisfying $H\leq Lv_0$, let
$\iota:\mathcal R_H\longrightarrow\mathcal C_H$ be a selection of
columns.  We denote by
\[
 \Delta_{H,L,\iota}
\]
the determinant of the actual selected arithmetic matrix, with the
logarithmic truncation factor $F=\theta^{-1}$.  This is precisely
\leanref{LemniscateExponentConditional}{QuadraticPi.LemniscateExponentConditional.minor}
(followed by \texttt{Matrix.det}), hence also
\nolinkurl{LemniscateDeterminant.selectedMinor}.  The base-coordinate truncation agrees coefficientwise on the retained
rows. The additive logarithm truncations are related to the full formal
jets by a unipotent filtered row change whose determinant is one. Thus
the same selected determinant equals the full-formal determinant, without
an entrywise equality assertion or an analytic approximation error. See
\leanref{LemniscateLogJetDeterminant}{QuadraticPi.LemniscateLogJetDeterminant.det_fullLog_selected_eq_selectedMinor}.

\subsection{The explicit theta multiplier}
\label{ana:subsec:theta}

The pole multiplier used below is a concrete theta series.  Put
\begin{equation}
 \Theta(\zeta)=\sum_{n\in\mathbb Z}
       \exp(-\pi n^2+2\pi\mathrm i n\zeta),\qquad
 \zeta(z)=\frac z\Omega-\frac{1+\mathrm i}{2},\qquad
 S(z)=\Theta(\zeta(z)).
 \label{ana:eq:thetaexplicit}
\end{equation}
The Gaussian decay gives locally uniform convergence of the series and
its derivatives, so $\Theta$ and $S$ are entire.  In Mathlib notation
$\Theta(\zeta)=\texttt{jacobiTheta\textsubscript{2}}(\zeta,\mathrm i)$;
this two-variable notation represents the displayed theta-three series,
not the theta function with characteristic $(1/2,0)$.
The project definitions and holomorphy theorem are
\leanref{LemniscateAnalyticTheta}{QuadraticPi.LemniscateAnalytic.squareTheta},
\leanref{LemniscateAnalyticTheta}{QuadraticPi.LemniscateAnalytic.squareTheta_differentiable},
and
\leanref{LemniscateAnalyticRealization}{QuadraticPi.LemniscateAnalytic.omegaMultiplier}.

Index shifts in the absolutely convergent series give
\begin{equation}
 \Theta(\zeta+1)=\Theta(\zeta),\qquad
 \Theta(\zeta+\mathrm i)=e^{\pi-2\pi\mathrm i\zeta}\Theta(\zeta),
 \qquad \Theta(-\zeta)=\Theta(\zeta).
 \label{ana:eq:thetaquasi}
\end{equation}
Thus $S(z+\Omega)=S(z)$ and
$S(z+\mathrm i\Omega)=e^{\pi-2\pi\mathrm i\zeta(z)}S(z)$.
Applying the second identity at $(1-\mathrm i)/2$, and using the first
and third identities, gives
$\Theta((1+\mathrm i)/2)=-\Theta((1+\mathrm i)/2)=0$.
Consequently $S$ vanishes at every point of
$\Lambda_\Omega=\Omega(\mathbb Z+\mathrm i\mathbb Z)$.
This vanishing is all that pole clearing needs; no assertion about the
complete zero divisor or the multiplicities of the zeros is used here.
The identities are
\leanref{LemniscateAnalyticTheta}{QuadraticPi.LemniscateAnalytic.squareTheta_add_one},
\leanref{LemniscateAnalyticTheta}{QuadraticPi.LemniscateAnalytic.squareTheta_add_I},
and
\leanref{LemniscateAnalyticTheta}{QuadraticPi.LemniscateAnalytic.squareTheta_half_zero};
the lattice conclusion is
\leanref{LemniscateAnalyticRealization}{QuadraticPi.LemniscateAnalytic.omegaMultiplier_lattice_zero}.

The Gaussian-normalized norm
\begin{equation}
 \mathcal N(\zeta)=e^{-\pi(\operatorname{Im}\zeta)^2}|\Theta(\zeta)|
 \label{ana:eq:gaussiannorm}
\end{equation}
is periodic under both $1$ and $\mathrm i$, since the norm of the second
multiplier in \eqref{ana:eq:thetaquasi} is
$e^{\pi+2\pi\operatorname{Im}\zeta}$.
Moreover
\[
 \Theta(0)=\sum_{n\in\mathbb Z}e^{-\pi n^2}>0.
\]
At $z_*=(1+\mathrm i)\Omega/2$ and every Gaussian translate, this gives
\begin{equation}
 |S(z_*+\Omega(a+b\mathrm i))|
     =e^{\pi b^2}\Theta(0)\geq\Theta(0)>0
       \qquad(a,b\in\mathbb Z).
 \label{ana:eq:thetagridlower}
\end{equation}
In particular there is a fixed $U\geq1$ for which all these values are
nonzero and their inverse $D$th powers have norm at most $U^D$.
The formal development proves the stronger assertion on a single fixed
closed disk about \emph{every} such translate.  To obtain it, continuity
and $\Theta(0)>0$ give $r,c>0$ with $|\Theta(\zeta)|\geq c$ when
$|\zeta|\leq r$.  Periodicity of \eqref{ana:eq:gaussiannorm} then yields
\[
 |\Theta(\zeta+a+b\mathrm i)|\geq
 e^{-\pi(\operatorname{Im}\zeta)^2}|\Theta(\zeta)|
       \geq e^{-\pi r^2}c.
\]
Rescale the disk by $\Omega$ and take a suitable
$U\geq\max(1,e^{\pi r^2}/c)$.
See
\leanref{LemniscateAnalyticTheta}{QuadraticPi.LemniscateAnalytic.gaussianThetaNorm_add_gaussian}
and
\leanref{LemniscateAnalyticRealization}{QuadraticPi.LemniscateAnalytic.omegaMultiplier_uniform_nonzero_inverse_pow}.

\begin{lemma}[Entire cleared coordinates and their growth]
\label{ana:lem:thetaclearing}
The functions $S^2\wp$ and $S^3\wp'/2$, initially defined away from
$\Lambda_\Omega$, extend to entire functions $X_c,Y_c$.  There are
fixed positive constants $C_S,C_X,C_Y$ such that, for every $z$,
\begin{equation}
 |S(z)|\leq C_S e^{\pi(\operatorname{Im}\zeta(z))^2},\quad
 |X_c(z)|\leq C_X e^{2\pi(\operatorname{Im}\zeta(z))^2},\quad
 |Y_c(z)|\leq C_Y e^{3\pi(\operatorname{Im}\zeta(z))^2}.
 \label{ana:eq:clearedgrowth}
\end{equation}
Consequently a single $C_g\geq1$ satisfies, whenever $2a+3b\leq D$,
\begin{equation}
 \left|S(z)^{D-2a-3b}X_c(z)^aY_c(z)^b\right|
 \leq C_g^D e^{\pi D(\operatorname{Im}\zeta(z))^2}
 \leq C_g^D e^{\pi D(|z|/\Omega+1)^2}.
 \label{ana:eq:uniformmonomialgrowth}
\end{equation}
\end{lemma}

\begin{proof}
At each lattice point $S$ has a zero of order at least one, whereas
$\wp$ has a pole of order two and $\wp'$ a pole of order three.
Thus the two products have removable singularities.  Their canonical
fillings are the functions in
\leanref{LemniscateAnalyticRealization}{QuadraticPi.LemniscateAnalytic.omegaClearedX_entire}
and
\leanref{LemniscateAnalyticRealization}{QuadraticPi.LemniscateAnalytic.omegaClearedY_entire}.

For $F=S,X_c,Y_c$ and respectively $n=1,2,3$, consider
$e^{-n\pi(\operatorname{Im}\zeta(z))^2}|F(z)|$.
Away from the lattice, this is $\mathcal N(\zeta(z))^n$ times the norm
of the corresponding elliptic factor $1,\wp,\wp'/2$.  It is therefore
lattice-periodic.  Its continuity after filling the singularities extends
this identity to all $z$.  It is bounded on a compact fundamental square,
and hence everywhere; this proves \eqref{ana:eq:clearedgrowth}.
This argument retains the coefficient $\pi$ for a single multiplier,
which is needed in the subsequent radius budget.  Its formal general
form is
\leanref{LemniscateAnalyticGrowth}{QuadraticPi.LemniscateAnalytic.omega_cleared_quadratic_growth}.

Take $C_g=\max(1,C_S,C_X,C_Y)$.  The exponent of the Gaussian after
multiplication is exactly
$D-2a-3b+2a+3b=D$, while the exponent of $C_g$ is
$D-2a-3b+a+b\leq D$.  This proves the first inequality in
\eqref{ana:eq:uniformmonomialgrowth}.  The second follows from
$|\operatorname{Im}\zeta(z)|\leq|z|/\Omega+1$.
The uniform formal theorem is
\leanref{LemniscateAnalyticColumns}{QuadraticPi.LemniscateAnalytic.omegaMonomialColumn_uniform_growth};
the radius estimate is
\leanref{LemniscateAnalyticTags}{QuadraticPi.LemniscateAnalytic.omegaThetaArgument_im_sq_le}.
\end{proof}

\subsection{The common multiplier and the determinant identity}
\label{ana:subsec:transport}

Let $\Lambda_\Omega=\Omega(\mathbb Z+\mathrm i\mathbb Z)$ and
$z_*=(1+\mathrm i)\Omega/2$.  The normalization proved earlier identifies
the Weierstrass functions of this lattice with the cubic
$y^2=x^3-x$, where $x=\wp$ and $y=\wp'/2$.
Write $S(z)$ for the entire multiplier
\texttt{LemniscateAnalytic.omegaMultiplier} and write $X_c,Y_c$ for its
entire cleared coordinates.  Away from the lattice,
\[
 X_c=S^2\wp,\qquad Y_c=S^3\wp'/2.
\]
For $2a+3b\leq D$ the entire function
\begin{equation}
 g_{D,a,b}(z)=S(z)^{D-2a-3b}X_c(z)^aY_c(z)^b
 \label{ana:eq:entiremonomial}
\end{equation}
satisfies $g_{D,a,b}=S^D\wp^a(\wp'/2)^b$ off the lattice.
The construction and equality are
\leanref{LemniscateAnalyticColumns}{QuadraticPi.LemniscateAnalytic.omegaMonomialColumn_entire}
and
\leanref{LemniscateAnalyticColumns}{QuadraticPi.LemniscateAnalytic.omegaMonomialColumn_eq}.
Use the same integer
\begin{equation}
 d_H=\lceil H/w_0\rceil
 \label{ana:eq:commonD}
\end{equation}
for every selected column.

The rational and true period centers differ by the constants
\begin{equation}
 \delta_{j,i}=\lambda_j(p_i/q_i-\Omega).
 \label{ana:eq:delta}
\end{equation}
For a column $c=(a,b,\alpha)$ define
\begin{equation}
 f_{c,\beta,\delta}(z)=g_{d_H,a,b}(z)
       \prod_{i=1}^m\binom{\alpha_i}{\beta_i}
                  (z-z_*+\delta_i)^{\alpha_i-\beta_i}.
 \label{ana:eq:displacedcol}
\end{equation}
As usual, the binomial coefficient makes a factor zero when
$\beta_i>\alpha_i$; the exponent in that case is interpreted as natural
subtraction, as in the formal source.
Let $\mathcal T_k(F)$ denote the coefficient of $u^k$ at zero,
equivalently its Cauchy integral on $|u|=1/2$ for the functions used below.
The cleared matrix associated with the \emph{same} selection is
\begin{equation}
 B_{H,\iota}(r,c)=
 \mathcal T_{k_r}\!\left(
 f_{\iota(c),\beta_r,\delta_j}
       (z_*+\Omega\lambda_j+\log(1+u))\right).
 \label{ana:eq:clearedmatrix}
\end{equation}
Here the logarithm is analytic on a neighbourhood of the closed
$1/2$-disk.  The entire functions in \eqref{ana:eq:displacedcol} make the
coefficient and contour definitions agree; no global choice of a
logarithm is required.

\begin{proposition}[Exact transport of the selected determinant]
\label{ana:prop:transport}
For every integer $H>0$, every $L$ with $H\leq Lv_0$, and every
selection $\iota$, one has
\begin{equation}
 |\det B_{H,\iota}|=
 \left(\prod_{j\in\{0,\ldots,N-1\}^2}
          |S(z_*+\Omega\lambda_j)|^{d_H|\mathcal B_H|}\right)
 |\Delta_{H,L,\iota}|.
 \label{ana:eq:exactnorm}
\end{equation}
There is a constant $U\geq1$, independent of $H,m,N,p,q$ and $\iota$,
such that every multiplier in this product is nonzero and
\begin{equation}
 |S(z_*+\Omega\lambda_j)^{-D}|\leq U^D\qquad(D\in\mathbb N).
 \label{ana:eq:inverseS}
\end{equation}
Consequently
\begin{equation}
 |\Delta_{H,L,\iota}|\leq U^{d_HM_H}|\det B_{H,\iota}|.
 \label{ana:eq:removemultiplier}
\end{equation}
\end{proposition}

\begin{proof}
First pass from the truncated logarithms to the full formal logarithm
by the determinant-one filtered row shear just described. The base
truncation is coefficientwise exact. Then the formal curve parameter is replaced by
$\psi(u)=y(z_*+\log(1+u))$.  Normalization gives
$\psi(0)=0$ and $\psi'(0)=-1$; the formal elliptic logarithm then
becomes $\log(1+u)$.  The induced map on each weighted jet packet is
triangular, with diagonal $(-1)^k$ at base order $k$, so its determinant
has absolute value one.  The same weighted packet is stable under
multiplication by the Taylor series of
$S(z_*+\Omega\lambda_j+\log(1+u))^{d_H}$.  Its diagonal is the constant
$S(z_*+\Omega\lambda_j)^{d_H}$, repeated $|\mathcal B_H|$ times.
Periodicity identifies the elliptic coordinates at the shifted centers,
and \eqref{ana:eq:delta} identifies the additive coordinates exactly.
Multiplying these triangular determinants proves \eqref{ana:eq:exactnorm}.
This also explains why selected columns remain unchanged.

The grid points are lattice translates of the fixed diagonal half-period.
The multiplier's uniform nonvanishing and inverse-power bound at these
translates gives \eqref{ana:eq:inverseS}.  Inverting the nonzero factors
in \eqref{ana:eq:exactnorm} proves \eqref{ana:eq:removemultiplier}.
The actual matrix identity and norm identity are
\leanref{LemniscateActualMinorTransport}{QuadraticPi.LemniscateActualMinorTransport.clearedMinor_eq_transformed}
and
\leanref{LemniscateActualMinorTransport}{QuadraticPi.LemniscateActualMinorTransport.norm_clearedMinor_eq}.
The uniform estimate is
\leanref{LemniscateMultiplierBudget}{QuadraticPi.LemniscateAnalytic.omegaInverseConstant_bound},
and its application to this matrix is
\leanref{LemniscateIrrationalityExponent}{QuadraticPi.LemniscateIrrationalityExponent.norm_actual_minor_le}.
\end{proof}

\subsection{Finite translation and collision estimates}
\label{ana:subsec:translation}

Let $\mathcal A_H=\{\tau\in\mathbb N^m:W(\tau)\leq H\}$ and
$T_H=|\mathcal A_H|$.  For a tag $\tau\in\mathcal A_H$ set
\begin{align}
 G_{c,\tau}(z)&=
 \left(\prod_i\binom{\alpha_i}{\tau_i}\right)
 g_{d_H,a,b}(z)(z-z_*)^{\sum_i(\alpha_i-\tau_i)},
 \label{ana:eq:tagcol}\\
 s(\beta,\tau,\delta)&=
 \prod_i\binom{\tau_i}{\beta_i}\delta_i^{\tau_i-\beta_i}.
 \label{ana:eq:rowscalar}
\end{align}
The finite binomial theorem gives
\begin{equation}
 f_{c,\beta,\delta}(z)=
     \sum_{\tau\in\mathcal A_H}s(\beta,\tau,\delta)G_{c,\tau}(z).
 \label{ana:eq:finiteexpansion}
\end{equation}
The support restriction follows from $\tau\leq\alpha$ and
$W(\alpha)\leq H$.  Thus the displacement constants are treated by an
exact finite identity, with no remainder series.
For an assignment $g:\mathcal R_H\to\mathcal A_H$ let
\[
 B_g(r,c)=\mathcal T_{k_r}\!\left(
       G_{\iota(c),g(r)}(z_*+\Omega\lambda_j+\log(1+u))\right).
\]
Multilinearity in rows yields
\begin{equation}
 \det B_{H,\iota}=
 \sum_{g:\mathcal R_H\to\mathcal A_H}
 \left(\prod_{r\in\mathcal R_H}
             s(\beta_r,g(r),\delta_j)\right)\det B_g.
 \label{ana:eq:detexpansion}
\end{equation}
These identities are
\leanref{LemniscateDisplacementExpansion}{QuadraticPi.LemniscateDisplacementExpansion.displacedColumn_expansion}
and
\leanref{LemniscateTranslatedDeterminant}{QuadraticPi.LemniscateTranslatedDeterminant.det_displacedMatrix_expansion}.

Assume $\nu>2$ and
\begin{equation}
 |\Omega-p_i/q_i|\leq q_i^{-\nu}\quad(1\leq i\leq m),
 \qquad 0<w_{\min}\leq\min_iw_i.
 \label{ana:eq:badapp}
\end{equation}
Put $C_\delta=\log(\sqrt2N)+\nu$.  Since
$|\lambda_j|\leq\sqrt2N$ and $w_i\leq\log q_i+1$,
\begin{equation}
 |\delta_{j,i}|\leq\exp(C_\delta-\nu w_i).
 \label{ana:eq:deltabound}
\end{equation}
If $s(\beta,\tau,\delta)\ne0$, then $\beta\leq\tau$ coordinatewise.
The binomial bound $\binom{n}{k}\leq2^n$, together with
$\sum_i\tau_i\leq H/w_{\min}$, gives
\begin{equation}
 |s(\beta,\tau,\delta)|\leq
 \exp\left(-\nu(W(\tau)-W(\beta))
               +\frac{H(\log2+C_\delta)}{w_{\min}}\right).
 \label{ana:eq:scalarbound}
\end{equation}
In particular, for a nonzero scalar product in
\eqref{ana:eq:detexpansion}, $\sum_rW(g(r))\geq M_HH\overline W_H$, and
\begin{equation}
 \left|\prod_rs(\beta_r,g(r),\delta_j)\right|
 \leq\exp\left(-\nu\left(\sum_rW(g(r))-M_HH\overline W_H\right)
                +\frac{M_HH(\log2+C_\delta)}{w_{\min}}\right).
 \label{ana:eq:scalarproduct}
\end{equation}
The precise statements are
\leanref{LemniscateDisplacementBounds}{QuadraticPi.LemniscateDisplacement.gridDelta_exp_bound},
\leanref{LemniscateDisplacementBounds}{QuadraticPi.LemniscateDisplacement.norm_rowScalar_exp_weighted},
and
\leanref{LemniscateTranslatedDeterminant}{QuadraticPi.LemniscateTranslatedDeterminant.scalar_product_bound}.

The entire growth constant $C_g\geq1$ can be fixed before every grid,
dimension, source monomial, and tag.  For $R\geq1$ and $|z|\leq R$,
the source-degree bounds imply
\begin{equation}
 |G_{c,\tau}(z)|\leq\exp\bigl(H\mathcal H_H(R)\bigr),\qquad
 \mathcal H_H(R)=\frac{d_H}{H}
   \left(\log C_g+\pi(R/\Omega+1)^2\right)
       +\frac{\log(2(R+|z_*|))}{w_{\min}}.
 \label{ana:eq:taggrowth}
\end{equation}
Indeed the elliptic factor contributes the first term, while the binomial
and polynomial factors are bounded by
$[2(R+|z_*|)]^{\sum_i\alpha_i}$ and
$\sum_i\alpha_i\leq H/w_{\min}$.
See
\leanref{LemniscateAnalyticTags}{QuadraticPi.LemniscateAnalytic.omegaTagColumn_uniform_growth}.
Set
\begin{equation}
 r_*=2\bigl(|z_*|+\sqrt2\Omega+3/4\bigr),\quad
 R_N=r_*N,\quad \kappa=\frac{\log2}{4},\quad s_0=2^{-1/2}.
 \label{ana:eq:radius}
\end{equation}
Then $R_N\geq1$ and
$|z_*+\Omega\lambda_j|+3/4\leq R_N/2$.

\begin{lemma}[Analytic collision bound]
\label{ana:lem:collision}
Let $n_\tau=|g^{-1}(\tau)|$.  Then
\begin{equation}
 \begin{aligned}
 |\det B_g|&\leq
 \exp\left[-\kappa\sum_\tau n_\tau^2+
 M_HH\left(\mathcal H_H(R_N)+\frac{\log2}{v_0}
                         +\rho_H\right)\right],\\
 \rho_H&=\frac{\log M_H+\kappa-\log(1-s_0)}{H}.
 \end{aligned}
 \label{ana:eq:collision}
\end{equation}
\end{lemma}

\begin{proof}
Expand each entire tag column in its Taylor series on the radius-$R_N$
disk.  Cauchy's estimate bounds the normalized Taylor coefficients by
$\exp(H\mathcal H_H(R_N))$.  For the translated logarithmic row
functional, the coefficient attached to degree $d$ has absolute value at
most $2^{k_r}(s_0^2)^d$.  The half-radius bound above and
$|\log(1+u)|\leq3/4$ on $|u|=1/2$ give this estimate.
In a nonzero determinant term, rows with the same tag must use distinct
Taylor degrees.  Their degree sum is at least $\binom{n_\tau}{2}$.
Splitting the geometric decay into two equal factors bounds the
determinant by
\[
 M_H!\,\exp(M_HH\mathcal H_H(R_N))
 2^{\sum_rk_r}\,
 s_0^{\sum_\tau\binom{n_\tau}{2}}
 (1-s_0)^{-M_H}.
\]
Absolute convergence justifies passage from finite Taylor sums to the
whole determinant.  Now use $M_H!\leq\exp(M_H\log M_H)$,
\eqref{ana:eq:rowbounds}, and
$2\sum_\tau\binom{n_\tau}{2}=\sum_\tau n_\tau^2-M_H$.
This proves \eqref{ana:eq:collision} with its displayed constant.
The instantiated lemma is
\leanref{LemniscateAnalyticCollision}{QuadraticPi.LemniscateAnalytic.omegaTaggedDeterminant_exp_bound};
the general analytic lemma is
\leanref{OAI/NumberTheory/PiExponent/Analysis/AnalyticDeterminantCollision}{OAI.PiExponent.AnalyticCollision.translated_row_determinant_exp_bound}.
\end{proof}

Fix additionally $A>0$ and $0<\eta<1$, and let
\begin{equation}
 L_H=|\{\tau\in\mathbb N^m:W(\tau)\leq AH\}|,
 \qquad c_H=\frac{\kappa\eta^2M_H}{HL_H}.
 \label{ana:eq:collisionrate}
\end{equation}
For each assignment in \eqref{ana:eq:detexpansion} there are two cases.
If $\sum_rW(g(r))\geq A(1-\eta)M_HH$, then
\eqref{ana:eq:scalarproduct} supplies the normalized negative term
$-\nu(A(1-\eta)-\overline W_H)$.  Otherwise more than $\eta M_H$ rows
have tags of weight at most $AH$.  At most $L_H$ tags can occur there,
so Cauchy--Schwarz gives
$\sum_\tau n_\tau^2\geq\eta^2M_H^2/L_H$.  In this second case the
scalar product loses no negative weight, since every nonzero such product
has $\sum_rW(g(r))\geq M_HH\overline W_H$; the collision bound therefore
supplies $-c_H$.
There are $T_H^{M_H}$ assignments.  Summing their bounds and taking
logarithms proves, when $\det B_{H,\iota}\ne0$,
\begin{align}
 \frac{\log|\det B_{H,\iota}|}{M_HH}
 \leq{}&\frac{\log2+C_\delta}{w_{\min}}+
   \mathcal H_H(R_N)+\frac{\log2}{v_0}+\rho_H
      +\frac{\log T_H}{H}\notag\\
 &+\max\{-c_H,-\nu(A(1-\eta)-\overline W_H)\}.
 \label{ana:eq:rawupper}
\end{align}
This is
\leanref{LemniscateTranslatedDeterminant}{QuadraticPi.LemniscateTranslatedDeterminant.displaced_determinant_log_bound}.

\subsection{Fixed constants and the finite-cutoff bound}
\label{ana:subsec:bound}

With $C_g,U,r_*$ fixed as above, put
\begin{equation}
 C_0=\log(C_gU)+\pi(r_*/\Omega+1)^2,\qquad
 C_w=4\sqrt2(r_*+|z_*|).
 \label{ana:eq:fixedconstants}
\end{equation}
Then $C_0\geq0$ and $C_w\geq1$.  These constants depend only on the
fixed lemniscate lattice and its chosen entire multiplier; in particular
they precede the choice of $\nu$, the dimension, the weights, the grid,
the rational approximations, and the selected columns.
They are exactly \texttt{fixedCost} and \texttt{fixedWeightCost} in
\leanref{LemniscateSameMinorBound}{QuadraticPi.LemniscateSameMinorBound.fixedCost}.
Define
\begin{align}
 E_{\rm an}&=\frac{C_0K}{w_0}+\frac{\log2}{v_0}
                  +\frac{\nu+\log(C_wK)}{w_{\min}},
 \label{ana:eq:anerror}\\
 e_H&=\frac{C_0K}{H}+\rho_H+\frac{\log T_H}{H}.
 \label{ana:eq:remainder}
\end{align}

\begin{proposition}[Upper bound for the literal arithmetic minor]
\label{ana:prop:upper}
Under \eqref{ana:eq:badapp}, for every integer $H>0$, every $L$ with
$H\leq Lv_0$, and every selection $\iota$ with
$\Delta_{H,L,\iota}\ne0$, one has
\begin{equation}
 \frac{\log|\Delta_{H,L,\iota}|}{M_HH}
 \leq E_{\rm an}+e_H+
          \max\{-c_H,-\nu(A(1-\eta)-\overline W_H)\}.
 \label{ana:eq:finalupper}
\end{equation}
For fixed finite data, $e_H\longrightarrow0$ as $H\longrightarrow\infty$.
\end{proposition}

\begin{proof}
Proposition~\ref{ana:prop:transport} implies that the cleared determinant
is nonzero and adds $(d_H/H)\log U$ to \eqref{ana:eq:rawupper}.
The estimates
\[
 \frac{d_H}{H}\leq\frac1{w_0}+\frac1H,
 \qquad
 \log(C_gU)+\pi(r_*N/\Omega+1)^2\leq C_0N^2
\]
bound the common-power cost.  The remaining radius logarithms satisfy
\[
 \log2+\log(\sqrt2N)+\log(2(r_*N+|z_*|))
       \leq\log(C_wN^2).
\]
These give exactly \eqref{ana:eq:finalupper}, using $N^2=K$.
This numerical consolidation is
\leanref{LemniscateAnalyticFinalBudget}{QuadraticPi.LemniscateAnalytic.final_holomorphic_budget};
the finite-cutoff assembly is
\leanref{LemniscateSameMinorBound}{QuadraticPi.LemniscateSameMinorBound.explicit_bound_of_norm_transfer}.

Weighted simplex counting, with all weights already fixed, gives
\begin{align}
 \frac{M_H}{H^{m+1}}&\longrightarrow
   \frac{K\theta^m}{(m+1)!v_0\prod_iw_i},
 &\frac{L_H}{H^m}&\longrightarrow
   \frac{A^m}{m!\prod_iw_i}.
 \label{ana:eq:counting}
\end{align}
The analogous count for $T_H$ has leading coefficient
$1/(m!\prod_iw_i)$.  All these coefficients are positive.  Hence
$\log M_H/H\to0$ and $\log T_H/H\to0$, proving $e_H\to0$.
Moreover,
\begin{equation}
 c_H\longrightarrow c_\infty
       =\frac{\kappa\eta^2K\theta^m}{(m+1)v_0A^m}.
 \label{ana:eq:climit}
\end{equation}
The counting results are
\leanref{OAI/NumberTheory/PiExponent/Approximation/MatrixCounting}{OAI.PiExponent.MatrixCounting.tendsto_rowCount_normalized}
and
\leanref{OAI/NumberTheory/PiExponent/Approximation/MatrixCounting}{OAI.PiExponent.MatrixCounting.tendsto_lowIndexCount_normalized};
the two limits used in the final assembly are
\leanref{LemniscateAnalyticRemainder}{QuadraticPi.LemniscateAnalyticRemainder.tendsto_remainder}
and
\leanref{LemniscateExponentConditional}{QuadraticPi.LemniscateExponentConditional.tendsto_collisionRate}.
\end{proof}

\begin{remark}
The formal interface
\texttt{LemniscateExponentConditional.SameMinorAnalyticStatement}
quantifies first over $\nu>2$ and admissible scalar data, then over
$p,q,w_{\min}$ satisfying \eqref{ana:eq:badapp}; it supplies a single
function $e(H)\to0$ valid eventually for \emph{every} allowed $L$ and
every nonzero selected minor.  Proposition~\ref{ana:prop:upper} supplies
the stronger finite-cutoff bound with the explicit
\eqref{ana:eq:remainder}.  Thus the interface is proved by
\leanref{LemniscateIrrationalityExponent}{QuadraticPi.LemniscateIrrationalityExponent.sameMinorAnalyticStatement}.
The word \texttt{Conditional} and the historical comments in the
intermediate module record its earlier role as a generic implication;
they do not represent an assumption in this final theorem.
\end{remark}

\section{Parameter selection and the irrationality exponent}
\label{ana:sec:exponent}

\subsection{Choosing the dimension before the denominators}
\label{ana:subsec:parameters}

Let $c_{\mathrm{lcm}}=\log4+4$ denote the fixed arithmetic constant
\texttt{OAI.PiExponent.Arithmetic.lcmConstant} occurring in the
denominator-clearing lower bound.  For a chosen $\nu>2$, choose rational
$\theta,A,B,C$ and real $\eta$ such that
\begin{gather}
 0<\theta<A<B<1,
 \quad 1<C,
 \quad B<C^{-1},
 \quad B<C\theta<1,
 \quad0<\eta<1,\label{ana:eq:params}\\
 G:=\nu\bigl(A(1-\eta)-\theta\bigr)-(1-\theta)>0,
 \qquad \varepsilon=\frac{\min(1,G)}4.
 \label{ana:eq:gapbudget}
\end{gather}
For completeness, these inequalities are simultaneously possible because
$\nu>2$.  Choose rational
$b_0\in(1/2,1-1/\nu)$ and sufficiently small rational $\delta>0$;
then $\theta=1-\delta$ and $A=1-b_0\delta$ satisfy
$A^2<\theta<A<1$ and $\nu(A-\theta)>1-\theta$.
Choose rational $C$ in $(A/\theta,1/A)$, and then rational
$B$ in $(A,\min(C^{-1},C\theta))$.
Finally take a sufficiently small positive $\eta$ to preserve the
strict exponent gap.  This is the parameter construction
\leanref{OAI/NumberTheory/PiExponent/Approximation/Parameters}{OAI.PiExponent.exists_parameters}.

For a dimension $m$ to be chosen large, set
\begin{equation}
 N=\lfloor\sqrt{C^m}\rfloor,
 \quad K=N^2,
 \quad w_0=B^{-m},
 \quad v_0=4K\theta^mB^{-m}.
 \label{ana:eq:squaredimension}
\end{equation}
In particular these weights are rational, $N\geq1$, and
\begin{equation}
 K\theta^m<1,
 \qquad K(w_0/v_0)\theta^m=\frac14<1.
 \label{ana:eq:epavolumes}
\end{equation}
These are the two volume conditions required by EPA.

\begin{lemma}[A square grid with strict margins]
\label{ana:lem:margin}
The dimension $m\geq1$ can be chosen so that
\begin{equation}
 \frac{c_{\mathrm{lcm}} m/\theta+\log2}{v_0}
               +\frac{C_0K}{w_0}<\varepsilon,
 \qquad c_\infty>3.
 \label{ana:eq:fixedmargin}
\end{equation}
\end{lemma}

\begin{proof}
For all sufficiently large $m$, $C^m\leq4K\leq4C^m$, whence
\[
 \frac{K}{w_0}\leq(CB)^m,
 \qquad v_0\geq(C\theta/B)^m.
\]
The inequalities in \eqref{ana:eq:params} give
$CB<1<C\theta/B$, so the first expression in
\eqref{ana:eq:fixedmargin} tends to zero.  Substituting
\eqref{ana:eq:squaredimension} in \eqref{ana:eq:climit} gives the exact
identity
\[
 c_\infty=\frac{\kappa\eta^2}{4(m+1)}(B/A)^m.
\]
Since $B/A>1$, this tends to infinity.  A single sufficiently large
dimension therefore gives both strict margins.
The floor-square estimates and limits are proved in
\leanref{LemniscateSquareParameters}{QuadraticPi.LemniscateSquareParameters.exists_squareDimension_margin};
the exact budget choice is
\leanref{LemniscateParameterMargin}{QuadraticPi.LemniscateParameterMargin.exists_margin}.
The complete package with the EPA base data is
\leanref{LemniscateDataConstruction}{QuadraticPi.LemniscateDataConstruction.exists_scalarData}.
\end{proof}

Fix this scalar data, and apply EPA to obtain its successive weight
thresholds.  Suppose for contradiction that bad approximations are
available with unbounded denominators:
\begin{equation}
 \forall Q\in\mathbb N\quad\exists p\in\mathbb Z, q\in\mathbb N,
 \quad q\geq Q,\quad |\Omega-p/q|\leq q^{-\nu}.
 \label{ana:eq:unboundedbad}
\end{equation}
Because $\Omega\ne0$ and $\nu>0$, taking $q$ sufficiently large forces
$p\ne0$.  Choose the $m$ approximations successively, so that the actual
weights $w_i=\lceil\log q_i\rceil$ exceed the EPA threshold depending
on the preceding weights.  They can at the same time all exceed any
fixed common lower bound.  This order of choice is proved in
\leanref{LemniscateDataConstruction}{QuadraticPi.LemniscateDataConstruction.exists_approximations_above_thresholds}.

The arithmetic error is
\begin{equation}
 E_{\rm ar}=\frac{c_{\mathrm{lcm}} m}{\theta v_0}
                 +c_{\mathrm{lcm}}\sum_i\frac1{w_i}+\frac{\theta}{w_{\min}}.
 \label{ana:eq:arerror}
\end{equation}
The already fixed first margin in \eqref{ana:eq:fixedmargin} covers
the terms independent of $w_{\min}$.  Now take the common lower bound
so large that
\[
 c_{\mathrm{lcm}}\sum_i\frac1{w_i}
       +\frac{\theta+\nu+\log(C_wK)}{w_{\min}}<\varepsilon.
\]
Then the same chosen approximations satisfy
\begin{equation}
 E_{\rm ar}+E_{\rm an}<2\varepsilon.
 \label{ana:eq:totalerror}
\end{equation}
This simultaneous selection, including all successive EPA thresholds,
is
\leanref{LemniscateDataConstruction}{QuadraticPi.LemniscateDataConstruction.exists_small_error_approximations}.
At this point $\nu,\theta,A,B,C,\eta,m,N,w_0,v_0,p,q,w_{\min}$
are all fixed; only the cutoff $H$, its permitted truncation level, and
the EPA-selected columns will vary.

\subsection{Comparison on cofinally many cutoffs}
\label{ana:subsec:contradiction}

\begin{theorem}[No unbounded family of overly accurate approximations]
\label{ana:thm:no-unbounded}
For every $\nu>2$, condition \eqref{ana:eq:unboundedbad} is false.
\end{theorem}

\begin{proof}
Choose the data and approximations just described.  EPA and the
denominator-clearing arithmetic estimate yield, for every cutoff lower
bound $H_0$, an integer $H\geq H_0$, an allowed $L$, and an injective
selection $\iota$ such that $\Delta=\Delta_{H,L,\iota}\ne0$ and
\begin{equation}
 -(1-\overline W_H)-E_{\rm ar}
        \leq V_H:=\frac{\log|\Delta|}{M_HH}.
 \label{ana:eq:arithlower}
\end{equation}
The precise unconditional arithmetic theorem is
\leanref{LemniscateUnconditionalMinors}{QuadraticPi.LemniscateUnconditionalMinors.cofinal_arithmetic_lower_bounds},
instantiated with $F=\theta^{-1}$.
Proposition~\ref{ana:prop:upper} applies to this very determinant and
gives
\begin{equation}
 V_H\leq E_{\rm an}+e_H+
       \max\{-c_H,-\nu(A(1-\eta)-\overline W_H)\}.
 \label{ana:eq:compareupper}
\end{equation}

By \eqref{ana:eq:gapbudget}, \eqref{ana:eq:fixedmargin}, and
\eqref{ana:eq:totalerror},
\[
 E_{\rm ar}+E_{\rm an}<G,
 \qquad 1+E_{\rm ar}+E_{\rm an}<c_\infty.
\]
Indeed $4\varepsilon\leq G$, $\varepsilon\leq1/4$, and
$c_\infty>3$.  The limits $e_H\to0$ and $c_H\to c_\infty$ imply
that for all sufficiently large real $H$,
\begin{equation}
 E_{\rm ar}+E_{\rm an}+e_H<G,
 \qquad 1+E_{\rm ar}+E_{\rm an}+e_H<c_H.
 \label{ana:eq:eventualmargins}
\end{equation}
Choose a cofinal integer cutoff large enough for these inequalities.
The collision branch in \eqref{ana:eq:compareupper} is then strictly
less than $-1-E_{\rm ar}$, which is at most
$-(1-\overline W_H)-E_{\rm ar}$ because $\overline W_H\geq0$.
For the displacement branch, put $x=A(1-\eta)$.  Since
$\overline W_H\leq\theta$ and $\nu>1$,
\[
 \nu(x-\overline W_H)-(1-\overline W_H)
 =\nu x-1-(\nu-1)\overline W_H
 \geq\nu x-1-(\nu-1)\theta=G.
\]
The first inequality in \eqref{ana:eq:eventualmargins} shows that this
branch too is strictly below the lower bound in
\eqref{ana:eq:arithlower}.  Both alternatives contradict the comparison.
Notice that the proof requires nonzero arithmetic minors only at
cofinally many integer cutoffs, not at every cutoff.

The scalar comparison, including this real-to-integer cofinal step, is
\leanref{CofinalComparison}{QuadraticPi.cofinal_determinant_bounds_inconsistent}.
The assembly is
\leanref{LemniscateExponentConditional}{QuadraticPi.LemniscateExponentConditional.no_unbounded_approximations_of_analyticStatement},
with its analytic argument supplied by the unconditional theorem
\nolinkurl{LemniscateIrrationalityExponent.sameMinorAnalyticStatement}.
\end{proof}

\begin{corollary}[Strict eventual lower bound]
\label{ana:cor:strict}
For every $\nu>2$ there is an integer $Q\geq2$ such that
\begin{equation}
 q^{-\nu}<|\Omega-p/q|
 \quad\text{for every }p\in\mathbb Z\text{ and every integer }q\geq Q.
 \label{ana:eq:eventuallower}
\end{equation}
\end{corollary}

\begin{proof}
The negation would provide \eqref{ana:eq:unboundedbad} above every
denominator bound.  Theorem~\ref{ana:thm:no-unbounded} excludes it.
The compiled unconditional declaration is
\leanref{LemniscateIrrationalityExponent}{QuadraticPi.Omega_eventual_rational_lower_bound}.
\end{proof}

\subsection{The ordinary irrationality exponent}
\label{ana:subsec:usualmu}

For a rational number $r$ let $\operatorname{den}(r)$ be its positive
reduced denominator.  The formal development uses the definition
\begin{equation}
 \mu(x)=\sup\left\{\nu>0:
 \begin{array}{l}
 \text{infinitely many }r\in\mathbb Q\text{ satisfy }
 \operatorname{den}(r)\geq2,\\[-2pt]
 0<|x-r|<\operatorname{den}(r)^{-\nu}
 \end{array}\right\}.
 \label{ana:eq:mudef}
\end{equation}
This is
\leanref{ExponentConclusion}{QuadraticPi.irrationalityExponent}.
The strict positivity of the approximation error causes no vacuity:
irrationality follows from the same all-denominator estimate.

\begin{theorem}[The lemniscate irrationality exponent]
\label{ana:thm:mu}
The integral-defined lemniscate period is irrational and satisfies
\[
 \boxed{\mu(\Omega)=2}.
\]
\end{theorem}

\begin{proof}
First apply \eqref{ana:eq:eventuallower} at $\nu=3$.  If $\Omega=a/b$
were rational with $b>0$, the exact representations
$\Omega=(an)/(bn)$, for arbitrarily large $n$, would contradict the
strict positive lower bound.  Thus $\Omega$ is irrational.  This step is
\leanref{ExponentConclusion}{QuadraticPi.irrational_of_eventualLowerBound},
and its unconditional instance is
\leanref{LemniscateIrrationalityExponent}{QuadraticPi.Omega_irrational}.

Dirichlet's approximation theorem supplies infinitely many reduced
rationals with $|\Omega-r|<\operatorname{den}(r)^{-2}$.  Irrationality
makes each error positive; discarding the finitely many possible
denominator-one approximations shows that $2$ belongs to the set in
\eqref{ana:eq:mudef}.  The precise step is
\leanref{ExponentConclusion}{QuadraticPi.two_mem_approximationExponents}.

Conversely, fix $\nu>2$.  By \eqref{ana:eq:eventuallower}, every good
reduced rational at this exponent has denominator less than the
corresponding $Q$.  Its approximation error is less than one.  For each
of the finitely many denominators the numerator is therefore bounded
between $(\Omega-1)q$ and $(\Omega+1)q$, giving only finitely many such
rationals.  Hence no $\nu>2$ belongs to the approximation-exponent
set.  The set is nonempty and bounded above, contains $2$, and has
supremum $2$.
This is proved abstractly in
\leanref{ExponentConclusion}{QuadraticPi.irrationalityExponent_eq_two_of_eventualLowerBound}.
Its final unconditional instance is
\leanref{LemniscateIrrationalityExponent}{QuadraticPi.irrationalityExponent_Omega_eq_two}.
The same result with \eqref{ana:eq:mudef} fully unfolded is
\leanref{LemniscateIrrationalityExponent}{QuadraticPi.Omega_approximation_supremum_eq_two}.
\end{proof}

The final theorem therefore has no EPA, normalization, analytic-bound,
or irrationality hypothesis: each is either proved on the path above or
derived in the concluding argument.  The historical conditional assembly
is retained as a reusable theorem, and its analytic premise is supplied
by the explicit same-minor bound.

\appendix
\section{Formal theorem, dependency audit, and reproduction}
\label{app:formal}

\subsection{The public endpoint and its proof term}

The final declaration is in
\nolinkurl{QuadraticPi/LemniscateIrrationalityExponent.lean}:
\begin{Verbatim}[fontsize=\small,breaklines=true,breakanywhere=true]
theorem irrationalityExponent_Omega_eq_two :
    irrationalityExponent Lemniscate.Omega = 2 :=
  irrationalityExponent_eq_two_of_analyticStatement
    fixedCost fixedWeightCost
    fixedCost_nonneg fixedWeightCost_one_le
    LemniscateIrrationalityExponent.sameMinorAnalyticStatement
\end{Verbatim}
This declaration lies in namespace \nolinkurl{QuadraticPi}. There are no
parameters or hypotheses in its full type. The supplied analytic statement
is itself a theorem, obtained from the exact determinant norm transport,
the explicit analytic estimate, and the vanishing remainder.

The source file
\nolinkurl{LemniscateExponentConditional.lean} preserves a useful conditional
assembly interface. Its historical comments describe that module's scope
when it was written. In the final module its analytic premise has a proved
witness; the period normalization is also a proved theorem. The current
endpoint does not retain either hypothesis. In particular the intermediate
proposition \nolinkurl{SameMinorAnalyticStatement} is not declared as an
axiom or inserted as an implicit instance.

\subsection{Definitions and foundations}

The period is exactly
\begin{Verbatim}[fontsize=\small,breaklines=true,breakanywhere=true]
def integrand (t : ℝ) : ℝ := 1 / Real.sqrt (1 - t ^ 4)
def Omega : ℝ := 2 * ∫ t in (0 : ℝ)..1, integrand t
\end{Verbatim}
These definitions use Lean's interval-integral notation. The exponent is the supremum of the set described in \eqref{eq:mu-main}, with
the denominator of the reduced rational number supplied by Lean's
\nolinkurl{Rat.den}. The audit prints both definitions and the intermediate
sets, so the name of the endpoint alone is not the identification of the
mathematical target. For $\Omega$ the approximation-exponent set is proved
nonempty and bounded above. Its real supremum therefore agrees with the
usual extended-real irrationality exponent; no assertion about infinite
exponents of arbitrary real numbers is encoded by this use of real supremum.

The endpoint's transitive dependency audit reports exactly
\begin{Verbatim}[fontsize=\small]
[propext, Classical.choice, Quot.sound]
\end{Verbatim}
These are propositional extensionality, classical choice, and quotient
soundness in Lean's standard classical foundations. No additional
mathematical axiom, \nolinkurl{sorryAx}, or unfinished proof is present in
the audited endpoint. As usual, the formal claim is relative to the Lean
kernel and its standard foundations. The mathematical content of the target
is also checked through its unfolded definition and the separate strict
all-denominator bound and irrationality theorem.

\subsection{Recorded verification and publication snapshot}

The publication snapshot is labelled
\nolinkurl{lemniscate-v1-2026-10-16}. It contains the complete local import
closure of \nolinkurl{QuadraticPi.LemniscateEPA6Complete}: 147
\nolinkurl{QuadraticPi} modules and 828 \nolinkurl{OAI} modules. All 975
mathematical source files are byte-for-byte copies of the verified project.
The pinned public dependencies are recorded in
\nolinkurl{lake-manifest.json}. Compiler binaries and build caches are not
included.

On October 16, 2026, the complete target was built again successfully with
5,814 jobs, and the final exponent module with 5,808 jobs. These totals
include cached dependency jobs. A new publication audit checked 316
distinct declarations, including all 154 declarations cited in this paper
and the retained completion-audit declarations. It printed their transitive
axioms and the definitions of the period and approximation exponent.
Every reported axiom belonged to
\nolinkurl{propext}, \nolinkurl{Classical.choice}, and
\nolinkurl{Quot.sound}. Both builds and the audit exited successfully.
This is a verification in the established pinned environment, using its
compatible cache; an independent build from an empty cache is not claimed.

The accompanying article-source archive places the proof snapshot in
\nolinkurl{anc/lean}. The files below are relative to that directory:
\begin{itemize}[leftmargin=*,itemsep=2pt]
\item \nolinkurl{validation/PublicationEndpointBuild.log} and
      \nolinkurl{validation/PublicationCompleteBuild.log};
\item \nolinkurl{validation/LemniscatePublicationAudit.lean} and
      \nolinkurl{validation/PublicationAudit.log};
\item \nolinkurl{validation/VERIFICATION.json}, recording the declaration
      inventory, per-declaration axiom sets, and validation scope;
\item \nolinkurl{SOURCE-MANIFEST.json}, recording file hashes, and
      \nolinkurl{README.txt}, giving the reproduction instructions.
\end{itemize}
The larger research backups preserve the earlier checkpoints and their
history. Reproduction of this article uses the accompanying publication
snapshot and its pinned public dependencies.

\subsection{Pinned environment and build commands}

\begin{center}
\begin{tabular}{@{}lp{0.72\linewidth}@{}}
\toprule
Component & Pinned version or commit\\
\midrule
Lean & 4.34.1\\
Lean commit & \nolinkurl{5045d0056413266e57c625dcd7c365b10e377c52}\\
Mathlib & \nolinkurl{d13f23b723b8a846827a245b89c10fc7d3f11612}\\
OAI source provenance & \nolinkurl{fd4aeeb2ee4fc729c18d98444fed42fd0529eeeb}\\
\bottomrule
\end{tabular}
\end{center}

After extracting the article sources, change into \nolinkurl{anc/lean}.
Install the toolchain named in \nolinkurl{lean-toolchain}, using an ordinary
Lean/elan installation. On Linux or WSL, the supplied Python 3 wrapper
limits compilation to two concurrent compiler processes and four threads
per compiler, while executing the unchanged installed Lean binary:
\begin{Verbatim}[fontsize=\small,breaklines=true,breakanywhere=true]
elan toolchain install leanprover/lean4:v4.34.1
python3 scripts/bounded_lake.py exe cache get
python3 scripts/bounded_lake.py build \
  QuadraticPi.LemniscateEPA6Complete
python3 scripts/bounded_lake.py env lean \
  validation/LemniscatePublicationAudit.lean
\end{Verbatim}
The cache-download step is optional and concerns the pinned public
dependencies. A build without that cache may take substantially longer.
Network access is required to obtain the pinned dependencies. Native
Windows validation is not claimed; use WSL for the supplied scheduler.

The underlying commands are ordinary
\nolinkurl{lake build QuadraticPi.LemniscateEPA6Complete} and
\nolinkurl{lake env lean validation/LemniscatePublicationAudit.lean}.
The publication copy sets the default Lake target to the completion
umbrella. Its \nolinkurl{README.txt} is the entry point; historical module
comments describe the scope of those modules when written.

\section{Exact Lean declaration index}
\label{app:lean-index}

Paths below are relative to \nolinkurl{anc/lean} in the accompanying
article-source archive. Line numbers locate declarations in the frozen
source files; the fully qualified name is the stable reference when
viewing a later revision. Every distinct declaration cited in this paper
is listed. The publication audit inventory also includes supporting
declarations not individually cited in the prose.

\begingroup
\footnotesize
\setlength{\LTleft}{0pt}
\setlength{\LTright}{0pt}
\renewcommand{\arraystretch}{1.08}
\begin{longtable}{@{}p{0.085\linewidth}>{\raggedright\arraybackslash}p{0.875\linewidth}@{}}
\toprule
Reference & Declaration and source location\\
\midrule
\endfirsthead
\toprule
Reference & Declaration and source location (continued)\\
\midrule
\endhead
\bottomrule
\endfoot
\hypertarget{lean.L001}{\textsf{L001}} & \nolinkurl{QuadraticPi.Omega_eventual_rational_lower_bound}\par
\nolinkurl{QuadraticPi/LemniscateIrrationalityExponent.lean}, line 52.\\[5pt]
\hypertarget{lean.L002}{\textsf{L002}} & \nolinkurl{QuadraticPi.Omega_irrational}\par
\nolinkurl{QuadraticPi/LemniscateIrrationalityExponent.lean}, line 60.\\[5pt]
\hypertarget{lean.L003}{\textsf{L003}} & \nolinkurl{QuadraticPi.irrationalityExponent_Omega_eq_two}\par
\nolinkurl{QuadraticPi/LemniscateIrrationalityExponent.lean}, line 65.\\[5pt]
\hypertarget{lean.L004}{\textsf{L004}} & \nolinkurl{QuadraticPi.Omega_approximation_supremum_eq_two}\par
\nolinkurl{QuadraticPi/LemniscateIrrationalityExponent.lean}, line 72.\\[5pt]
\hypertarget{lean.L005}{\textsf{L005}} & \nolinkurl{QuadraticPi.Lemniscate.Omega}\par
\nolinkurl{QuadraticPi/LemniscateTarget.lean}, line 27.\\[5pt]
\hypertarget{lean.L006}{\textsf{L006}} & \nolinkurl{QuadraticPi.Lemniscate.integrand_intervalIntegrable}\par
\nolinkurl{QuadraticPi/LemniscateTarget.lean}, line 68.\\[5pt]
\hypertarget{lean.L007}{\textsf{L007}} & \nolinkurl{QuadraticPi.Lemniscate.Omega_pos}\par
\nolinkurl{QuadraticPi/LemniscateTarget.lean}, line 80.\\[5pt]
\hypertarget{lean.L008}{\textsf{L008}} & \nolinkurl{QuadraticPi.LemniscatePeriod.squarePair}\par
\nolinkurl{QuadraticPi/LemniscatePeriodGeometry.lean}, line 22.\\[5pt]
\hypertarget{lean.L009}{\textsf{L009}} & \nolinkurl{QuadraticPi.LemniscateSquareLattice.square_G_conj}\par
\nolinkurl{QuadraticPi/LemniscateSquareLattice.lean}, line 126.\\[5pt]
\hypertarget{lean.L010}{\textsf{L010}} & \nolinkurl{QuadraticPi.LemniscateSquareLattice.square_weierstrassP_ofReal_re}\par
\nolinkurl{QuadraticPi/LemniscateSquareLattice.lean}, line 98.\\[5pt]
\hypertarget{lean.L011}{\textsf{L011}} & \nolinkurl{QuadraticPi.LemniscateSquareLattice.square_derivWeierstrassP_ofReal_re}\par
\nolinkurl{QuadraticPi/LemniscateSquareLattice.lean}, line 120.\\[5pt]
\hypertarget{lean.L012}{\textsf{L012}} & \nolinkurl{QuadraticPi.LemniscatePeriod.g₃_eq_zero}\par
\nolinkurl{QuadraticPi/LemniscatePeriodGeometry.lean}, line 93.\\[5pt]
\hypertarget{lean.L013}{\textsf{L013}} & \nolinkurl{QuadraticPi.LemniscatePeriod.weierstrassP_mul_I}\par
\nolinkurl{QuadraticPi/LemniscatePeriodGeometry.lean}, line 82.\\[5pt]
\hypertarget{lean.L014}{\textsf{L014}} & \nolinkurl{QuadraticPi.LemniscateSquareLattice.square_g₂_scale}\par
\nolinkurl{QuadraticPi/LemniscateSquareLattice.lean}, line 175.\\[5pt]
\hypertarget{lean.L015}{\textsf{L015}} & \nolinkurl{QuadraticPi.LemniscateSquareInvariantPositivity.G_four_re_lower_bound}\par
\nolinkurl{QuadraticPi/LemniscateSquareInvariantPositivity.lean}, line 102.\\[5pt]
\hypertarget{lean.L016}{\textsf{L016}} & \nolinkurl{QuadraticPi.LemniscateSquareInvariantPositivity.unit_g₂_re_pos}\par
\nolinkurl{QuadraticPi/LemniscateSquareInvariantPositivity.lean}, line 112.\\[5pt]
\hypertarget{lean.L017}{\textsf{L017}} & \nolinkurl{QuadraticPi.LemniscateSquareNormalization.normalizingSide_positive}\par
\nolinkurl{QuadraticPi/LemniscateSquareNormalization.lean}, line 45.\\[5pt]
\hypertarget{lean.L018}{\textsf{L018}} & \nolinkurl{QuadraticPi.LemniscateSquareNormalization.normalizedPair_g₂}\par
\nolinkurl{QuadraticPi/LemniscateSquareNormalization.lean}, line 52.\\[5pt]
\hypertarget{lean.L019}{\textsf{L019}} & \nolinkurl{QuadraticPi.LemniscateSquareNormalization.normalizedPair_g₃}\par
\nolinkurl{QuadraticPi/LemniscateSquareNormalization.lean}, line 55.\\[5pt]
\hypertarget{lean.L020}{\textsf{L020}} & \nolinkurl{QuadraticPi.LemniscateDifferential.order_derivWeierstrassP}\par
\nolinkurl{QuadraticPi/LemniscateDifferential.lean}, line 26.\\[5pt]
\hypertarget{lean.L021}{\textsf{L021}} & \nolinkurl{QuadraticPi.LemniscateDifferential.deriv_derivWeierstrassP}\par
\nolinkurl{QuadraticPi/LemniscateDifferential.lean}, line 66.\\[5pt]
\hypertarget{lean.L022}{\textsf{L022}} & \nolinkurl{QuadraticPi.LemniscateODEUniqueness.second_order_series_unique}\par
\nolinkurl{QuadraticPi/LemniscateODEUniqueness.lean}, line 13.\\[5pt]
\hypertarget{lean.L023}{\textsf{L023}} & \nolinkurl{QuadraticPi.LemniscateODEUniqueness.analytic_second_order_unique}\par
\nolinkurl{QuadraticPi/LemniscateODEUniqueness.lean}, line 68.\\[5pt]
\hypertarget{lean.L024}{\textsf{L024}} & \nolinkurl{QuadraticPi.LemniscateCriticalPoints.two_mul_mem_lattice_of_deriv_eq_zero}\par
\nolinkurl{QuadraticPi/LemniscateCriticalPoints.lean}, line 45.\\[5pt]
\hypertarget{lean.L025}{\textsf{L025}} & \nolinkurl{QuadraticPi.LemniscateRealBranch.real_weierstrassP_tendsto_atTop}\par
\nolinkurl{QuadraticPi/LemniscateRealBranch.lean}, line 38.\\[5pt]
\hypertarget{lean.L026}{\textsf{L026}} & \nolinkurl{QuadraticPi.LemniscateRealBranch.square_real_deriv_ne_zero}\par
\nolinkurl{QuadraticPi/LemniscateRealBranch.lean}, line 25.\\[5pt]
\hypertarget{lean.L027}{\textsf{L027}} & \nolinkurl{QuadraticPi.LemniscateRealBranch.one_lt_realP}\par
\nolinkurl{QuadraticPi/LemniscateRealBranch.lean}, line 87.\\[5pt]
\hypertarget{lean.L028}{\textsf{L028}} & \nolinkurl{QuadraticPi.LemniscateRealBranch.realP_halfPeriod_eq_one}\par
\nolinkurl{QuadraticPi/LemniscateRealBranch.lean}, line 101.\\[5pt]
\hypertarget{lean.L029}{\textsf{L029}} & \nolinkurl{QuadraticPi.LemniscateRealBranch.realDerivP_neg}\par
\nolinkurl{QuadraticPi/LemniscateRealBranch.lean}, line 124.\\[5pt]
\hypertarget{lean.L030}{\textsf{L030}} & \nolinkurl{QuadraticPi.LemniscateRealCalculus.hasDerivAt_realP}\par
\nolinkurl{QuadraticPi/LemniscateRealCalculus.lean}, line 25.\\[5pt]
\hypertarget{lean.L031}{\textsf{L031}} & \nolinkurl{QuadraticPi.LemniscateRealCalculus.hasDerivAt_realDerivP_of_normalization}\par
\nolinkurl{QuadraticPi/LemniscateRealCalculus.lean}, line 45.\\[5pt]
\hypertarget{lean.L032}{\textsf{L032}} & \nolinkurl{QuadraticPi.LemniscateRealCalculus.real_cubic_of_normalization}\par
\nolinkurl{QuadraticPi/LemniscateRealCalculus.lean}, line 59.\\[5pt]
\hypertarget{lean.L033}{\textsf{L033}} & \nolinkurl{QuadraticPi.Lemniscate.reciprocal_sqrt_hasDerivAt}\par
\nolinkurl{QuadraticPi/LemniscateReciprocalBranch.lean}, line 49.\\[5pt]
\hypertarget{lean.L034}{\textsf{L034}} & \nolinkurl{QuadraticPi.Lemniscate.halfPeriod_eq_integral_of_branch}\par
\nolinkurl{QuadraticPi/LemniscateInverseIntegral.lean}, line 41.\\[5pt]
\hypertarget{lean.L035}{\textsf{L035}} & \nolinkurl{QuadraticPi.Lemniscate.period_eq_Omega_of_real_branch}\par
\nolinkurl{QuadraticPi/LemniscateReciprocalBranch.lean}, line 90.\\[5pt]
\hypertarget{lean.L036}{\textsf{L036}} & \nolinkurl{QuadraticPi.LemniscateOmegaNormalization.omegaPair_g₂_eq_four}\par
\nolinkurl{QuadraticPi/LemniscateOmegaNormalization.lean}, line 35.\\[5pt]
\hypertarget{lean.L037}{\textsf{L037}} & \nolinkurl{QuadraticPi.LemniscateOmegaNormalization.period_eq_Omega_of_normalization}\par
\nolinkurl{QuadraticPi/LemniscateOmegaNormalization.lean}, line 16.\\[5pt]
\hypertarget{lean.L038}{\textsf{L038}} & \nolinkurl{QuadraticPi.LemniscateOmegaNormalization.normalizingSide_eq_Omega}\par
\nolinkurl{QuadraticPi/LemniscateOmegaNormalization.lean}, line 30.\\[5pt]
\hypertarget{lean.L039}{\textsf{L039}} & \nolinkurl{QuadraticPi.LemniscatePeriod.omega_diagonalHalfPeriod_not_mem}\par
\nolinkurl{QuadraticPi/LemniscatePeriodGeometry.lean}, line 145.\\[5pt]
\hypertarget{lean.L040}{\textsf{L040}} & \nolinkurl{QuadraticPi.LemniscatePeriod.omega_weierstrassP_diagonalHalfPeriod}\par
\nolinkurl{QuadraticPi/LemniscatePeriodGeometry.lean}, line 148.\\[5pt]
\hypertarget{lean.L041}{\textsf{L041}} & \nolinkurl{QuadraticPi.LemniscatePeriod.omega_derivWeierstrassP_diagonalHalfPeriod}\par
\nolinkurl{QuadraticPi/LemniscatePeriodGeometry.lean}, line 163.\\[5pt]
\hypertarget{lean.L042}{\textsf{L042}} & \nolinkurl{QuadraticPi.LemniscateLocal.localX}\par
\nolinkurl{QuadraticPi/LemniscateLocal.lean}, line 40.\\[5pt]
\hypertarget{lean.L043}{\textsf{L043}} & \nolinkurl{QuadraticPi.LemniscateLocal.ellipticLog}\par
\nolinkurl{QuadraticPi/LemniscateLocal.lean}, line 69.\\[5pt]
\hypertarget{lean.L044}{\textsf{L044}} & \nolinkurl{QuadraticPi.LemniscateLocal.localX_equation}\par
\nolinkurl{QuadraticPi/LemniscateLocal.lean}, line 45.\\[5pt]
\hypertarget{lean.L045}{\textsf{L045}} & \nolinkurl{QuadraticPi.LemniscateLocal.ellipticLog_derivative_equation}\par
\nolinkurl{QuadraticPi/LemniscateLocal.lean}, line 90.\\[5pt]
\hypertarget{lean.L046}{\textsf{L046}} & \nolinkurl{QuadraticPi.LemniscateLocal.ellipticLog_coeff_degree_cleared}\par
\nolinkurl{QuadraticPi/LemniscateLocal.lean}, line 95.\\[5pt]
\hypertarget{lean.L047}{\textsf{L047}} & \nolinkurl{QuadraticPi.LemniscateTaylorBridge.omegaXSeries}\par
\nolinkurl{QuadraticPi/LemniscateTaylorBridge.lean}, line 53.\\[5pt]
\hypertarget{lean.L048}{\textsf{L048}} & \nolinkurl{QuadraticPi.LemniscateTaylorBridge.omegaYSeries}\par
\nolinkurl{QuadraticPi/LemniscateTaylorBridge.lean}, line 58.\\[5pt]
\hypertarget{lean.L049}{\textsf{L049}} & \nolinkurl{QuadraticPi.LemniscateOmegaGerm.omega_series_cubic}\par
\nolinkurl{QuadraticPi/LemniscateOmegaGerm.lean}, line 23.\\[5pt]
\hypertarget{lean.L050}{\textsf{L050}} & \nolinkurl{QuadraticPi.LemniscateOmegaGerm.omegaYSeries_derivative}\par
\nolinkurl{QuadraticPi/LemniscateOmegaGerm.lean}, line 27.\\[5pt]
\hypertarget{lean.L051}{\textsf{L051}} & \nolinkurl{QuadraticPi.LemniscateOmegaGerm.omega_localX_subst_eq}\par
\nolinkurl{QuadraticPi/LemniscateOmegaGerm.lean}, line 31.\\[5pt]
\hypertarget{lean.L052}{\textsf{L052}} & \nolinkurl{QuadraticPi.LemniscateOmegaGerm.omega_ellipticLog_subst_eq}\par
\nolinkurl{QuadraticPi/LemniscateOmegaGerm.lean}, line 35.\\[5pt]
\hypertarget{lean.L053}{\textsf{L053}} & \nolinkurl{QuadraticPi.LemniscateOmegaGerm.omegaYSeries_linear}\par
\nolinkurl{QuadraticPi/LemniscateOmegaGerm.lean}, line 39.\\[5pt]
\hypertarget{lean.L054}{\textsf{L054}} & \nolinkurl{QuadraticPi.LemniscateTaylorBridge.omegaXSeries_derivative}\par
\nolinkurl{QuadraticPi/LemniscateTaylorBridge.lean}, line 94.\\[5pt]
\hypertarget{lean.L055}{\textsf{L055}} & \nolinkurl{QuadraticPi.LemniscateLogParameterChart.logParameter}\par
\nolinkurl{QuadraticPi/LemniscateLogParameterChart.lean}, line 17.\\[5pt]
\hypertarget{lean.L056}{\textsf{L056}} & \nolinkurl{QuadraticPi.LemniscateLogParameterChart.xFunction}\par
\nolinkurl{QuadraticPi/LemniscateLogParameterChart.lean}, line 19.\\[5pt]
\hypertarget{lean.L057}{\textsf{L057}} & \nolinkurl{QuadraticPi.LemniscateLogParameterChart.yFunction}\par
\nolinkurl{QuadraticPi/LemniscateLogParameterChart.lean}, line 20.\\[5pt]
\hypertarget{lean.L058}{\textsf{L058}} & \nolinkurl{QuadraticPi.LemniscateLogParameterChart.xSeries}\par
\nolinkurl{QuadraticPi/LemniscateLogParameterChart.lean}, line 21.\\[5pt]
\hypertarget{lean.L059}{\textsf{L059}} & \nolinkurl{QuadraticPi.LemniscateLogParameterChart.ySeries}\par
\nolinkurl{QuadraticPi/LemniscateLogParameterChart.lean}, line 22.\\[5pt]
\hypertarget{lean.L060}{\textsf{L060}} & \nolinkurl{QuadraticPi.LemniscateLogParameterChart.ellSeries}\par
\nolinkurl{QuadraticPi/LemniscateLogParameterChart.lean}, line 23.\\[5pt]
\hypertarget{lean.L061}{\textsf{L061}} & \nolinkurl{QuadraticPi.LemniscateOmegaGerm.log_series_cubic}\par
\nolinkurl{QuadraticPi/LemniscateOmegaGerm.lean}, line 42.\\[5pt]
\hypertarget{lean.L062}{\textsf{L062}} & \nolinkurl{QuadraticPi.LemniscateOmegaGerm.log_ySeries_derivative}\par
\nolinkurl{QuadraticPi/LemniscateOmegaGerm.lean}, line 47.\\[5pt]
\hypertarget{lean.L063}{\textsf{L063}} & \nolinkurl{QuadraticPi.LemniscateOmegaGerm.log_localX_subst}\par
\nolinkurl{QuadraticPi/LemniscateOmegaGerm.lean}, line 56.\\[5pt]
\hypertarget{lean.L064}{\textsf{L064}} & \nolinkurl{QuadraticPi.LemniscateOmegaGerm.log_ellipticLog_subst}\par
\nolinkurl{QuadraticPi/LemniscateOmegaGerm.lean}, line 61.\\[5pt]
\hypertarget{lean.L065}{\textsf{L065}} & \nolinkurl{QuadraticPi.LemniscateOmegaGerm.log_ySeries_linear}\par
\nolinkurl{QuadraticPi/LemniscateOmegaGerm.lean}, line 53.\\[5pt]
\hypertarget{lean.L066}{\textsf{L066}} & \nolinkurl{QuadraticPi.LemniscateLogParameterChart.ellSeries_linear}\par
\nolinkurl{QuadraticPi/LemniscateLogParameterChart.lean}, line 68.\\[5pt]
\hypertarget{lean.L067}{\textsf{L067}} & \nolinkurl{QuadraticPi.SeparatedWeightPA.BaseData}\par
\nolinkurl{QuadraticPi/SeparatedWeightPA.lean}, line 23.\\[5pt]
\hypertarget{lean.L068}{\textsf{L068}} & \nolinkurl{QuadraticPi.LemniscateInterpolation.formalJet_cubic_relation}\par
\nolinkurl{QuadraticPi/LemniscateInterpolation.lean}, line 75.\\[5pt]
\hypertarget{lean.L069}{\textsf{L069}} & \nolinkurl{QuadraticPi.LemniscateInterpolation.Statement}\par
\nolinkurl{QuadraticPi/LemniscateInterpolation.lean}, line 136.\\[5pt]
\hypertarget{lean.L070}{\textsf{L070}} & \nolinkurl{QuadraticPi.LemniscateGlobalInterpolation.ellipticSeparatedWeightPA}\par
\nolinkurl{QuadraticPi/LemniscateGlobalInterpolationTheorem.lean}, line 37.\\[5pt]
\hypertarget{lean.L071}{\textsf{L071}} & \nolinkurl{QuadraticPi.LemniscatePlaceContact.contact}\par
\nolinkurl{QuadraticPi/LemniscatePlaceContact.lean}, line 90.\\[5pt]
\hypertarget{lean.L072}{\textsf{L072}} & \nolinkurl{QuadraticPi.LemniscateGrowthBudget.strict_growth_contact_margin}\par
\nolinkurl{QuadraticPi/LemniscateGrowthBudget.lean}, line 42.\\[5pt]
\hypertarget{lean.L073}{\textsf{L073}} & \nolinkurl{QuadraticPi.LemniscateGrowthBudget.growth_factor_lt_three_halves}\par
\nolinkurl{QuadraticPi/LemniscateGrowthBudget.lean}, line 52.\\[5pt]
\hypertarget{lean.L074}{\textsf{L074}} & \nolinkurl{QuadraticPi.LemniscateComparisonConstant.exists_thresholds}\par
\nolinkurl{QuadraticPi/LemniscateComparisonConstant.lean}, line 96.\\[5pt]
\hypertarget{lean.L075}{\textsf{L075}} & \nolinkurl{QuadraticPi.LemniscateWeightGeometry.NumericData}\par
\nolinkurl{QuadraticPi/LemniscateWeightGeometry.lean}, line 26.\\[5pt]
\hypertarget{lean.L076}{\textsf{L076}} & \nolinkurl{QuadraticPi.LemniscateGeometryFrame.formalJet_polynomialFrame}\par
\nolinkurl{QuadraticPi/LemniscateGeometryFrame.lean}, line 224.\\[5pt]
\hypertarget{lean.L077}{\textsf{L077}} & \nolinkurl{QuadraticPi.LemniscateWeightedDerivation.polynomialFrame_supportBound}\par
\nolinkurl{QuadraticPi/LemniscateWeightedDerivation.lean}, line 179.\\[5pt]
\hypertarget{lean.L078}{\textsf{L078}} & \nolinkurl{QuadraticPi.LemniscateAuxiliaryPolynomial.eventually_exists_nonzero_elliptic_auxiliary}\par
\nolinkurl{QuadraticPi/LemniscateAuxiliaryPolynomial.lean}, line 228.\\[5pt]
\hypertarget{lean.L079}{\textsf{L079}} & \nolinkurl{QuadraticPi.LemniscateCurveAuxiliary.exists_auxiliary_with_persistence_inputs}\par
\nolinkurl{QuadraticPi/LemniscateCurveAuxiliary.lean}, line 27.\\[5pt]
\hypertarget{lean.L080}{\textsf{L080}} & \nolinkurl{QuadraticPi.LemniscateNormalComparison.normal_product_comparison}\par
\nolinkurl{QuadraticPi/LemniscateNormalComparison.lean}, line 49.\\[5pt]
\hypertarget{lean.L081}{\textsf{L081}} & \nolinkurl{QuadraticPi.LemniscateDifferentialRigidity.finite_extension_no_exact_cubicDifferential}\par
\nolinkurl{QuadraticPi/LemniscateDifferentialRigidity.lean}, line 124.\\[5pt]
\hypertarget{lean.L082}{\textsf{L082}} & \nolinkurl{QuadraticPi.LemniscateCurveRigidity.constant_projection}\par
\nolinkurl{QuadraticPi/LemniscateCurveRigidity.lean}, line 20.\\[5pt]
\hypertarget{lean.L083}{\textsf{L083}} & \nolinkurl{QuadraticPi.LemniscateCurveInequality.weighted_curve_inequality}\par
\nolinkurl{QuadraticPi/LemniscateCurveInequality.lean}, line 18.\\[5pt]
\hypertarget{lean.L084}{\textsf{L084}} & \nolinkurl{QuadraticPi.LemniscateFixedJetScale.exists_scale}\par
\nolinkurl{QuadraticPi/LemniscateFixedJetScale.lean}, line 34.\\[5pt]
\hypertarget{lean.L085}{\textsf{L085}} & \nolinkurl{QuadraticPi.LemniscateJetSelector.localized_selector_eq}\par
\nolinkurl{QuadraticPi/LemniscateJetSelector.lean}, line 63.\\[5pt]
\hypertarget{lean.L086}{\textsf{L086}} & \nolinkurl{QuadraticPi.LemniscateJetSelector.selectorIdeal_le_prime_iff}\par
\nolinkurl{QuadraticPi/LemniscateJetSelector.lean}, line 74.\\[5pt]
\hypertarget{lean.L087}{\textsf{L087}} & \nolinkurl{QuadraticPi.LemniscateCurveModel.hyperplane_degree}\par
\nolinkurl{QuadraticPi/LemniscateCurveHyperplaneDegree.lean}, line 27.\\[5pt]
\hypertarget{lean.L088}{\textsf{L088}} & \nolinkurl{QuadraticPi.LemniscateCurveModel.exceptional_degree}\par
\nolinkurl{QuadraticPi/LemniscateCurveExceptionalDegree.lean}, line 50.\\[5pt]
\hypertarget{lean.L089}{\textsf{L089}} & \nolinkurl{QuadraticPi.LemniscateBlowupGeometry.interpolationBundle_ample}\par
\nolinkurl{QuadraticPi/LemniscateBlowupAmple.lean}, line 83.\\[5pt]
\hypertarget{lean.L090}{\textsf{L090}} & \nolinkurl{QuadraticPi.LemniscateGlobalInterpolation.actual_blowupBundle_ample}\par
\nolinkurl{QuadraticPi/LemniscateGlobalInterpolationTheorem.lean}, line 19.\\[5pt]
\hypertarget{lean.L091}{\textsf{L091}} & \nolinkurl{QuadraticPi.LemniscateGlobalInterpolation.eventual_packetMap_of_ample}\par
\nolinkurl{QuadraticPi/LemniscateGlobalInterpolation.lean}, line 78.\\[5pt]
\hypertarget{lean.L092}{\textsf{L092}} & \nolinkurl{QuadraticPi.LemniscateGlobalInterpolation.eventual_divisible_packetMap_of_ample}\par
\nolinkurl{QuadraticPi/LemniscateGlobalInterpolation.lean}, line 110.\\[5pt]
\hypertarget{lean.L093}{\textsf{L093}} & \nolinkurl{QuadraticPi.LemniscateArithmetic.coefficient_monomialImage}\par
\nolinkurl{QuadraticPi/LemniscateArithmetic.lean}, line 54.\\[5pt]
\hypertarget{lean.L094}{\textsf{L094}} & \nolinkurl{QuadraticPi.LemniscateInterpolationMatrix.formal_packets_surjective_iff_truncated}\par
\nolinkurl{QuadraticPi/LemniscateInterpolationMatrix.lean}, line 201.\\[5pt]
\hypertarget{lean.L095}{\textsf{L095}} & \nolinkurl{QuadraticPi.LemniscateInterpolationMatrix.matrix_surjective_of_epa_packet}\par
\nolinkurl{QuadraticPi/LemniscateInterpolationMatrix.lean}, line 416.\\[5pt]
\hypertarget{lean.L096}{\textsf{L096}} & \nolinkurl{QuadraticPi.LemniscateUnconditionalMinors.cofinal_minors}\par
\nolinkurl{QuadraticPi/LemniscateUnconditionalMinors.lean}, line 17.\\[5pt]
\hypertarget{lean.L097}{\textsf{L097}} & \nolinkurl{QuadraticPi.LemniscateLocal.ellipticLog_trunc_lcm_cleared}\par
\nolinkurl{QuadraticPi/LemniscateLocal.lean}, line 115.\\[5pt]
\hypertarget{lean.L098}{\textsf{L098}} & \nolinkurl{QuadraticPi.LemniscateDeterminant.selectedMinor_entries_gaussian}\par
\nolinkurl{QuadraticPi/LemniscateDeterminant.lean}, line 169.\\[5pt]
\hypertarget{lean.L099}{\textsf{L099}} & \nolinkurl{QuadraticPi.LemniscateDeterminant.selectedMinor_clearing_bound}\par
\nolinkurl{QuadraticPi/LemniscateDeterminant.lean}, line 215.\\[5pt]
\hypertarget{lean.L100}{\textsf{L100}} & \nolinkurl{QuadraticPi.LemniscateDeterminant.selectedMinor_arithmetic_lower_bound}\par
\nolinkurl{QuadraticPi/LemniscateDeterminant.lean}, line 242.\\[5pt]
\hypertarget{lean.L101}{\textsf{L101}} & \nolinkurl{QuadraticPi.LemniscateUnconditionalMinors.cofinal_arithmetic_lower_bounds}\par
\nolinkurl{QuadraticPi/LemniscateUnconditionalMinors.lean}, line 36.\\[5pt]
\hypertarget{lean.L102}{\textsf{L102}} & \nolinkurl{QuadraticPi.LemniscateAnalyticRowData.row_order_sum_le}\par
\nolinkurl{QuadraticPi/LemniscateAnalyticRowData.lean}, line 71.\\[5pt]
\hypertarget{lean.L103}{\textsf{L103}} & \nolinkurl{QuadraticPi.LemniscateExponentConditional.mean_bounds}\par
\nolinkurl{QuadraticPi/LemniscateExponentConditional.lean}, line 79.\\[5pt]
\hypertarget{lean.L104}{\textsf{L104}} & \nolinkurl{QuadraticPi.LemniscateExponentConditional.minor}\par
\nolinkurl{QuadraticPi/LemniscateExponentConditional.lean}, line 39.\\[5pt]
\hypertarget{lean.L105}{\textsf{L105}} & \nolinkurl{QuadraticPi.LemniscateLogJetDeterminant.det_fullLog_selected_eq_selectedMinor}\par
\nolinkurl{QuadraticPi/LemniscateLogJetDeterminant.lean}, line 99.\\[5pt]
\hypertarget{lean.L106}{\textsf{L106}} & \nolinkurl{QuadraticPi.LemniscateAnalytic.squareTheta}\par
\nolinkurl{QuadraticPi/LemniscateAnalyticTheta.lean}, line 20.\\[5pt]
\hypertarget{lean.L107}{\textsf{L107}} & \nolinkurl{QuadraticPi.LemniscateAnalytic.squareTheta_differentiable}\par
\nolinkurl{QuadraticPi/LemniscateAnalyticTheta.lean}, line 22.\\[5pt]
\hypertarget{lean.L108}{\textsf{L108}} & \nolinkurl{QuadraticPi.LemniscateAnalytic.omegaMultiplier}\par
\nolinkurl{QuadraticPi/LemniscateAnalyticRealization.lean}, line 25.\\[5pt]
\hypertarget{lean.L109}{\textsf{L109}} & \nolinkurl{QuadraticPi.LemniscateAnalytic.squareTheta_add_one}\par
\nolinkurl{QuadraticPi/LemniscateAnalyticTheta.lean}, line 26.\\[5pt]
\hypertarget{lean.L110}{\textsf{L110}} & \nolinkurl{QuadraticPi.LemniscateAnalytic.squareTheta_add_I}\par
\nolinkurl{QuadraticPi/LemniscateAnalyticTheta.lean}, line 29.\\[5pt]
\hypertarget{lean.L111}{\textsf{L111}} & \nolinkurl{QuadraticPi.LemniscateAnalytic.squareTheta_half_zero}\par
\nolinkurl{QuadraticPi/LemniscateAnalyticTheta.lean}, line 36.\\[5pt]
\hypertarget{lean.L112}{\textsf{L112}} & \nolinkurl{QuadraticPi.LemniscateAnalytic.omegaMultiplier_lattice_zero}\par
\nolinkurl{QuadraticPi/LemniscateAnalyticRealization.lean}, line 37.\\[5pt]
\hypertarget{lean.L113}{\textsf{L113}} & \nolinkurl{QuadraticPi.LemniscateAnalytic.gaussianThetaNorm_add_gaussian}\par
\nolinkurl{QuadraticPi/LemniscateAnalyticTheta.lean}, line 131.\\[5pt]
\hypertarget{lean.L114}{\textsf{L114}} & \nolinkurl{QuadraticPi.LemniscateAnalytic.omegaMultiplier_uniform_nonzero_inverse_pow}\par
\nolinkurl{QuadraticPi/LemniscateAnalyticRealization.lean}, line 67.\\[5pt]
\hypertarget{lean.L115}{\textsf{L115}} & \nolinkurl{QuadraticPi.LemniscateAnalytic.omegaClearedX_entire}\par
\nolinkurl{QuadraticPi/LemniscateAnalyticRealization.lean}, line 40.\\[5pt]
\hypertarget{lean.L116}{\textsf{L116}} & \nolinkurl{QuadraticPi.LemniscateAnalytic.omegaClearedY_entire}\par
\nolinkurl{QuadraticPi/LemniscateAnalyticRealization.lean}, line 43.\\[5pt]
\hypertarget{lean.L117}{\textsf{L117}} & \nolinkurl{QuadraticPi.LemniscateAnalytic.omega_cleared_quadratic_growth}\par
\nolinkurl{QuadraticPi/LemniscateAnalyticGrowth.lean}, line 66.\\[5pt]
\hypertarget{lean.L118}{\textsf{L118}} & \nolinkurl{QuadraticPi.LemniscateAnalytic.omegaMonomialColumn_uniform_growth}\par
\nolinkurl{QuadraticPi/LemniscateAnalyticColumns.lean}, line 48.\\[5pt]
\hypertarget{lean.L119}{\textsf{L119}} & \nolinkurl{QuadraticPi.LemniscateAnalytic.omegaThetaArgument_im_sq_le}\par
\nolinkurl{QuadraticPi/LemniscateAnalyticTags.lean}, line 34.\\[5pt]
\hypertarget{lean.L120}{\textsf{L120}} & \nolinkurl{QuadraticPi.LemniscateAnalytic.omegaMonomialColumn_entire}\par
\nolinkurl{QuadraticPi/LemniscateAnalyticColumns.lean}, line 19.\\[5pt]
\hypertarget{lean.L121}{\textsf{L121}} & \nolinkurl{QuadraticPi.LemniscateAnalytic.omegaMonomialColumn_eq}\par
\nolinkurl{QuadraticPi/LemniscateAnalyticColumns.lean}, line 24.\\[5pt]
\hypertarget{lean.L122}{\textsf{L122}} & \nolinkurl{QuadraticPi.LemniscateActualMinorTransport.clearedMinor_eq_transformed}\par
\nolinkurl{QuadraticPi/LemniscateActualMinorTransport.lean}, line 23.\\[5pt]
\hypertarget{lean.L123}{\textsf{L123}} & \nolinkurl{QuadraticPi.LemniscateActualMinorTransport.norm_clearedMinor_eq}\par
\nolinkurl{QuadraticPi/LemniscateActualMinorTransport.lean}, line 71.\\[5pt]
\hypertarget{lean.L124}{\textsf{L124}} & \nolinkurl{QuadraticPi.LemniscateAnalytic.omegaInverseConstant_bound}\par
\nolinkurl{QuadraticPi/LemniscateMultiplierBudget.lean}, line 30.\\[5pt]
\hypertarget{lean.L125}{\textsf{L125}} & \nolinkurl{QuadraticPi.LemniscateIrrationalityExponent.norm_actual_minor_le}\par
\nolinkurl{QuadraticPi/LemniscateIrrationalityExponent.lean}, line 18.\\[5pt]
\hypertarget{lean.L126}{\textsf{L126}} & \nolinkurl{QuadraticPi.LemniscateDisplacementExpansion.displacedColumn_expansion}\par
\nolinkurl{QuadraticPi/LemniscateDisplacementExpansion.lean}, line 88.\\[5pt]
\hypertarget{lean.L127}{\textsf{L127}} & \nolinkurl{QuadraticPi.LemniscateTranslatedDeterminant.det_displacedMatrix_expansion}\par
\nolinkurl{QuadraticPi/LemniscateTranslatedDeterminant.lean}, line 37.\\[5pt]
\hypertarget{lean.L128}{\textsf{L128}} & \nolinkurl{QuadraticPi.LemniscateDisplacement.gridDelta_exp_bound}\par
\nolinkurl{QuadraticPi/LemniscateDisplacementBounds.lean}, line 101.\\[5pt]
\hypertarget{lean.L129}{\textsf{L129}} & \nolinkurl{QuadraticPi.LemniscateDisplacement.norm_rowScalar_exp_weighted}\par
\nolinkurl{QuadraticPi/LemniscateDisplacementBounds.lean}, line 60.\\[5pt]
\hypertarget{lean.L130}{\textsf{L130}} & \nolinkurl{QuadraticPi.LemniscateTranslatedDeterminant.scalar_product_bound}\par
\nolinkurl{QuadraticPi/LemniscateTranslatedDeterminant.lean}, line 98.\\[5pt]
\hypertarget{lean.L131}{\textsf{L131}} & \nolinkurl{QuadraticPi.LemniscateAnalytic.omegaTagColumn_uniform_growth}\par
\nolinkurl{QuadraticPi/LemniscateAnalyticTags.lean}, line 143.\\[5pt]
\hypertarget{lean.L132}{\textsf{L132}} & \nolinkurl{QuadraticPi.LemniscateAnalytic.omegaTaggedDeterminant_exp_bound}\par
\nolinkurl{QuadraticPi/LemniscateAnalyticCollision.lean}, line 54.\\[5pt]
\hypertarget{lean.L133}{\textsf{L133}} & \nolinkurl{OAI.PiExponent.AnalyticCollision.translated_row_determinant_exp_bound}\par
\nolinkurl{OAI/NumberTheory/PiExponent/Analysis/AnalyticDeterminantCollision.lean}, line 56.\\[5pt]
\hypertarget{lean.L134}{\textsf{L134}} & \nolinkurl{QuadraticPi.LemniscateTranslatedDeterminant.displaced_determinant_log_bound}\par
\nolinkurl{QuadraticPi/LemniscateTranslatedDeterminant.lean}, line 194.\\[5pt]
\hypertarget{lean.L135}{\textsf{L135}} & \nolinkurl{QuadraticPi.LemniscateSameMinorBound.fixedCost}\par
\nolinkurl{QuadraticPi/LemniscateSameMinorBound.lean}, line 17.\\[5pt]
\hypertarget{lean.L136}{\textsf{L136}} & \nolinkurl{QuadraticPi.LemniscateAnalytic.final_holomorphic_budget}\par
\nolinkurl{QuadraticPi/LemniscateAnalyticFinalBudget.lean}, line 43.\\[5pt]
\hypertarget{lean.L137}{\textsf{L137}} & \nolinkurl{QuadraticPi.LemniscateSameMinorBound.explicit_bound_of_norm_transfer}\par
\nolinkurl{QuadraticPi/LemniscateSameMinorBound.lean}, line 41.\\[5pt]
\hypertarget{lean.L138}{\textsf{L138}} & \nolinkurl{OAI.PiExponent.MatrixCounting.tendsto_rowCount_normalized}\par
\nolinkurl{OAI/NumberTheory/PiExponent/Approximation/MatrixCounting.lean}, line 31.\\[5pt]
\hypertarget{lean.L139}{\textsf{L139}} & \nolinkurl{OAI.PiExponent.MatrixCounting.tendsto_lowIndexCount_normalized}\par
\nolinkurl{OAI/NumberTheory/PiExponent/Approximation/MatrixCounting.lean}, line 64.\\[5pt]
\hypertarget{lean.L140}{\textsf{L140}} & \nolinkurl{QuadraticPi.LemniscateAnalyticRemainder.tendsto_remainder}\par
\nolinkurl{QuadraticPi/LemniscateAnalyticRemainder.lean}, line 60.\\[5pt]
\hypertarget{lean.L141}{\textsf{L141}} & \nolinkurl{QuadraticPi.LemniscateExponentConditional.tendsto_collisionRate}\par
\nolinkurl{QuadraticPi/LemniscateExponentConditional.lean}, line 61.\\[5pt]
\hypertarget{lean.L142}{\textsf{L142}} & \nolinkurl{QuadraticPi.LemniscateIrrationalityExponent.sameMinorAnalyticStatement}\par
\nolinkurl{QuadraticPi/LemniscateIrrationalityExponent.lean}, line 40.\\[5pt]
\hypertarget{lean.L143}{\textsf{L143}} & \nolinkurl{OAI.PiExponent.exists_parameters}\par
\nolinkurl{OAI/NumberTheory/PiExponent/Approximation/Parameters.lean}, line 114.\\[5pt]
\hypertarget{lean.L144}{\textsf{L144}} & \nolinkurl{QuadraticPi.LemniscateSquareParameters.exists_squareDimension_margin}\par
\nolinkurl{QuadraticPi/LemniscateSquareParameters.lean}, line 184.\\[5pt]
\hypertarget{lean.L145}{\textsf{L145}} & \nolinkurl{QuadraticPi.LemniscateParameterMargin.exists_margin}\par
\nolinkurl{QuadraticPi/LemniscateParameterMargin.lean}, line 35.\\[5pt]
\hypertarget{lean.L146}{\textsf{L146}} & \nolinkurl{QuadraticPi.LemniscateDataConstruction.exists_scalarData}\par
\nolinkurl{QuadraticPi/LemniscateDataConstruction.lean}, line 40.\\[5pt]
\hypertarget{lean.L147}{\textsf{L147}} & \nolinkurl{QuadraticPi.LemniscateDataConstruction.exists_approximations_above_thresholds}\par
\nolinkurl{QuadraticPi/LemniscateDataConstruction.lean}, line 99.\\[5pt]
\hypertarget{lean.L148}{\textsf{L148}} & \nolinkurl{QuadraticPi.LemniscateDataConstruction.exists_small_error_approximations}\par
\nolinkurl{QuadraticPi/LemniscateDataConstruction.lean}, line 129.\\[5pt]
\hypertarget{lean.L149}{\textsf{L149}} & \nolinkurl{QuadraticPi.cofinal_determinant_bounds_inconsistent}\par
\nolinkurl{QuadraticPi/CofinalComparison.lean}, line 10.\\[5pt]
\hypertarget{lean.L150}{\textsf{L150}} & \nolinkurl{QuadraticPi.LemniscateExponentConditional.no_unbounded_approximations_of_analyticStatement}\par
\nolinkurl{QuadraticPi/LemniscateExponentConditional.lean}, line 114.\\[5pt]
\hypertarget{lean.L151}{\textsf{L151}} & \nolinkurl{QuadraticPi.irrationalityExponent}\par
\nolinkurl{QuadraticPi/ExponentConclusion.lean}, line 33.\\[5pt]
\hypertarget{lean.L152}{\textsf{L152}} & \nolinkurl{QuadraticPi.irrational_of_eventualLowerBound}\par
\nolinkurl{QuadraticPi/ExponentConclusion.lean}, line 43.\\[5pt]
\hypertarget{lean.L153}{\textsf{L153}} & \nolinkurl{QuadraticPi.two_mem_approximationExponents}\par
\nolinkurl{QuadraticPi/ExponentConclusion.lean}, line 92.\\[5pt]
\hypertarget{lean.L154}{\textsf{L154}} & \nolinkurl{QuadraticPi.irrationalityExponent_eq_two_of_eventualLowerBound}\par
\nolinkurl{QuadraticPi/ExponentConclusion.lean}, line 131.\\[5pt]
\end{longtable}
\endgroup


\begin{thebibliography}{9}
\bibitem{coxhyde}
D. A. Cox and T. Hyde,
\emph{The Galois theory of the lemniscate},
Journal of Number Theory \textbf{135} (2014), 43--59.
\url{https://doi.org/10.1016/j.jnt.2013.08.006}.

\bibitem{lean4}
L. de Moura and S. Ullrich,
\emph{The Lean 4 theorem prover and programming language},
in Automated Deduction--CADE 28, Lecture Notes in Computer Science
\textbf{12699}, Springer, 2021, pp. 625--635.
\url{https://doi.org/10.1007/978-3-030-79876-5_37}.

\bibitem{mathlib}
The mathlib Community,
\emph{The Lean mathematical library},
Proceedings of the 9th ACM SIGPLAN International Conference on
Certified Programs and Proofs (CPP 2020), ACM, 2020.
\url{https://doi.org/10.1145/3372885.3373824}.

\bibitem{oaiformalsnapshot}
OpenAI,
\emph{Mathematical manuscripts and supporting Lean proof artifacts},
repository \nolinkurl{openai/math}, source snapshot
\nolinkurl{fd4aeeb2ee4fc729c18d98444fed42fd0529eeeb} (2026).
\url{https://github.com/openai/math}.
The selected local source files, per-file hashes, provenance notice,
and upstream license are included in the accompanying Lean snapshot.

\bibitem{epa6artifact}
G. V. Chudnovsky, with OpenAI ChatGPT assistance,
\emph{Lemniscate-period formal proof snapshot},
mathematical development October 10, 2026; publication snapshot
\nolinkurl{lemniscate-v1-2026-10-16}.
The sources and verification records accompany this article in
\nolinkurl{anc/lean}; see Appendix~\ref{app:formal}.
\end{thebibliography}
\end{document}